\documentclass{amsart}

\usepackage{marktext}
\usepackage{gimac}

\usepackage{algpseudocode, algorithm}
\usepackage{mdframed}
\usepackage{algorithmicx}
\usepackage{algpseudocode}

\newmdenv[
  linewidth=0.5pt,
  skipabove=10pt,
  skipbelow=10pt,
  backgroundcolor=gray!3,
  linecolor=gray!60,
  innertopmargin=8pt,
  innerbottommargin=8pt,
  roundcorner=4pt,
  nobreak=false  
]{algbox}

\newcommand{\tr}{\mathrm{Tr}}

\newcommand{\Unif}{\mathrm{Unif}}
\newcommand{\gap}{\mathrm{Gap}}

\newcommand{\pt}{\mathrm{pt}}
\newcommand{\TV}{\mathrm{TV}}
\DeclareMathOperator{\dist}{dist}

\DeclarePairedDelimiter{\ip}{\langle}{\rangle}

\newtheorem{assumption}[theorem]{Assumption}
\begin{document}

\title[Rapid convergence of parallel and simlulated tempering]
{Rapid convergence of tempering chains to multimodal Gibbs measures}

\author[Son]{Seungjae Son}
\address{%
  Department of Mathematical Sciences, Carnegie Mellon University, Pittsburgh, PA 15213.
}
\email{seungjas@andrew.cmu.edu}

\begin{abstract}
We study the spectral gaps of parallel and simulated tempering chains targeting multimodal Gibbs measures. In particular, we consider chains constructed from Metropolis random walks that preserve the Gibbs distributions at a sequence of harmonically spaced temperatures. We prove that their spectral gaps admit polynomial lower bounds of order~$11$ and~$12$ in terms of the low target temperature. The analysis applies to a broad class of potentials, beyond mixture models, without requiring explicit structural information on the energy landscape. The main idea is to decompose the state space and construct a Lyapunov function based on a suitably perturbed potential, which allows us to establish lower bounds on the local spectral gaps.
\end{abstract}

\maketitle
\section{Introduction}
We show that parallel and simulated tempering for multimodal Gibbs measures exhibit a spectral gap that is polynomial in the target temperature. To this end, we first provide a brief description of the chains and state the main result, highlighting some of its key features and the main challenges in its proof in Section~\ref{s:brief-main-result}. We then survey the related literature and discuss the motivation and background in Section~\ref{s:motivation-and-background}.

\subsection{Informal statement of main result}\label{s:brief-main-result}

Let~$\T^d \cong [0,1)^d$ be the~$d$-dimensional torus and let~$U\colon \T^d \to [0,\infty)$ be an energy potential. For~$\epsilon > 0$, define the Gibbs distribution~$\pi_\epsilon$ with density
\begin{equation}\label{e:gibbs-d-def}
    \pi_\epsilon(x) = \frac{1}{Z_\epsilon} \tilde \pi_\epsilon (x)
    \quad\text{where}\quad 
    \tilde\pi_\epsilon(x) = \exp\paren*{-\frac{U(x)}{\epsilon}},
    \quad 
    Z_\epsilon = \int_{\T^d} \tilde \pi_\epsilon(y)\,dy\,.
\end{equation}

We are interested in sampling from~$\pi_\epsilon$ in the low-temperature regime (i.e.\ small~$\epsilon$). When~$U$ has multiple wells, standard Markov chain Monte Carlo methods often mix slowly due to metastability. A widely used approach to mitigate this issue is \emph{parallel tempering} (see, e.g., \cite{SwendsenWang86, Geyer91}), which simulates multiple chains at different temperatures and enables efficient transitions via swap moves.

We briefly describe one step of the parallel tempering chain. Let~$\epsilon_0 > \epsilon_1 > \cdots > \epsilon_N$ be a sequence of temperatures, and let~$(h_k)_{k=0}^N$ be the corresponding step sizes. The state space is~$\paren*{\T^d}^{N+1}$. Given the current state
\begin{equation}
    X = (X_0, X_1, \ldots, X_N)\,,
\end{equation}
one step of the chain proceeds as follows:
\begin{enumerate}[\indent 1.]
    \item \textbf{Swap move.}  
    With probability~$1/2$, do nothing. Otherwise, choose~$I\in \set{0, \ldots, N-1}$ uniformly at random and propose to swap~$X_I$ and~$X_{I+1}$. Accept the swap with probability
    \begin{equation}
        \min\set*{1, \frac{\tilde \pi_{\epsilon_I}(X_{I+1}) \tilde\pi_{\epsilon_{I+1}}(X_I)}{\tilde \pi_{\epsilon_I}(X_I)\tilde\pi_{\epsilon_{I+1}}(X_{I+1})}}\,.
    \end{equation}
    Denote the resulting state by~$X^{(1)}$.

    \item \textbf{Metropolis update.}  
    With probability~$1/2$, do nothing. Otherwise, choose $J\in \set{0, \ldots, N}$ uniformly at random. Sample~$\zeta \sim \Unif(B(0,1))$ and propose
    \begin{equation}
        Y_J = X^{(1)}_J + h_J \zeta\,.
    \end{equation}
    Accept this proposal with probability
    \begin{equation}
        \min\set*{1, \frac{\tilde \pi_{\epsilon_J}(Y_J)}{\tilde \pi_{\epsilon_J}(X^{(1)}_J)}}\,.
    \end{equation}
    Denote the resulting state by~$X^{(2)}$.

    \item \textbf{Swap move.}  
    Repeat the swap step starting from~$X^{(2)}$, and denote the final state by~$X^{\mathrm{new}}$.
\end{enumerate}

This defines a non-negative definite reversible Markov chain~$P_\pt$ on~$\paren*{\T^d}^{N+1}$ with invariant distribution~$\pi_\pt = \prod_{k=0}^N \pi_{\epsilon_k}$. A precise definition is given in Section~\ref{s:main-result} (see Definitions~\ref{d:pt-def} and~\ref{d:metropolis-def}).

To quantify convergence, we use the notion of spectral gap.

\begin{definition}[Spectral gap]
Let~$P$ be a Markov kernel on a Polish space~$\mathcal X$, reversible with respect to a probability measure~$\pi$. The spectral gap of~$P$ is
\begin{equation}\label{e:gap-def}
    \gap(P) = \inf_{f\in L^2(\pi)\setminus\set{0}} \frac{\mathcal E(f)}{\var_\pi(f)}\,,
\end{equation}
where
\begin{equation}
    \mathcal E(f) = \ip*{f , (I-P)f}_{L^2(\pi)} 
    = \frac{1}{2}\int_{\mathcal X}\int_{\mathcal X} \abs*{f(y)-f(x)}^2 P(x,dy)\pi(dx)\,.
\end{equation}
\end{definition}

We now state the main result informally.

\begin{theorem}\label{t:main-theorem-intro}
Let~$U\colon \T^d \to \R$ be a regular double-well potential with wells of equal depth (but not necessarily the same shape). 
Then there exist constants $\eta, c_1, c_2, \bar C_\mathrm{BV}, c_d > 0$ such that the following holds.

For any~$0 < \underline \epsilon < \bar \epsilon < 1$, set
\begin{equation}
    N = \ceil*{\frac{1}{\underline \epsilon}}, 
    \quad \epsilon_0 = \bar \epsilon, 
    \quad \epsilon_N = \underline \epsilon,
\end{equation}
and choose~$\paren*{1/\epsilon_k}_{k=0}^N$ to be linearly spaced. For each~$k$, define
\begin{equation}
    h_k = \min\set*{\eta \epsilon_k^2, 1}\,.
\end{equation}
Then the corresponding parallel tempering chain~$P_\pt$ satisfies
\begin{equation}
    \gap(P_\pt) \geq D \min\set*{c_2,\, c_1 \underline\epsilon^7} \underline\epsilon^4\,,
\end{equation}
where
\begin{equation}\label{e:D-def-intro}
    D(d, \epsilon_0) = c_d \exp\paren*{-\bar C_\mathrm{BV} - 2\paren*{1 + \frac{1}{\epsilon_0}}\norm{U}_{L^\infty}} h_0^2\,.
\end{equation}
\end{theorem}

The spectral gap bound for a non-negative definite reversible Markov chain implies quantitative convergence in total variation. The following corollary is an immediate consequence of~\cite[Theorem 2.1]{RobertsRosenthal97}.

\begin{corollary}
Under the assumptions and choices of Theorem~\ref{t:main-theorem-intro}, let~$\nu_0$ be an initial distribution on~$\mathcal X_\pt = (\T^d)^{N+1}$ such that~$\nu_0 \ll \pi_\pt$. Then, for all~$m\in \N$,
\begin{equation}
    \norm*{\nu_0 P_\pt^m - \pi_\pt}_\TV 
    \leq \frac{1}{2}\paren*{1 - \gap(P_\pt)}^m 
    \norm*{\frac{d\nu_0}{d\pi_\pt} - 1}_{L^2(\pi_\pt)}\,,
\end{equation}
where~$\gap(P_\pt)$ satisfies the lower bound in Theorem~\ref{t:main-theorem-intro}.
\end{corollary}

We note that the estimates used in the proof of Theorem~\ref{t:main-theorem-intro} can also be applied to the \emph{simulated tempering} chain~$P_{\mathrm{st}}$. We briefly describe one step of this chain. Let~$\epsilon_0 > \epsilon_1 > \cdots > \epsilon_N$ be a sequence of temperatures, and let~$(h_k)_{k=0}^N$ be the corresponding step sizes. The state space is~$\T^d\times \set{0, \ldots, N}$. Given the current state~$(Z, I)$, one step of the chain proceeds as follows:
\begin{enumerate}[\indent 1.]
    \item \textbf{Temperature update.}  
    With probability~$1/2$, do nothing. Otherwise, choose~$J \in \set{0, \ldots, N}$ with probability
    \begin{equation}
        \frac{\tilde \pi_{\epsilon_J}(Z)}{\sum_{k=0}^N \tilde \pi_{\epsilon_k}(Z)}\,,
    \end{equation}
    and move to~$(Z, J)$. 

    \item \textbf{Metropolis update.}  
    With probability~$1/2$, do nothing. Otherwise, sample~$\zeta \sim \Unif(B(0,1))$ and propose
    \begin{equation}
        Y = Z + h_{J} \zeta\,.
    \end{equation}
    Accept this proposal with probability
    \begin{equation}
        \min\set*{1, \frac{\tilde \pi_{\epsilon_{J}}(Y)}{\tilde \pi_{\epsilon_{J}}(Z)}}\,.
    \end{equation}
    Denote the resulting state by~$(Z^{(1)}, J)$.

    \item \textbf{Temperature update.}  
    Repeat the temperature update step starting from $(Z^{(1)}, J)$, and denote the final state by~$(Z^{(1)}, J^{(1)})$.
\end{enumerate}

This defines a non-negative definite reversible Markov chain~$P_{\mathrm{st}}$ on~$\T^d \times \set{0, \ldots, N}$ with invariant distribution~$\pi_{\mathrm{st}}(z,i) \propto \tilde \pi_{\epsilon_i}(z)$. In the simulated tempering setting, one obtains a bound analogous to that for parallel tempering, with the only difference being an additional factor of the final temperature in the order, as stated in the following. 

\begin{corollary}\label{c:simulated-tempering}
Under the same setting and with the same constants as in Theorem~\ref{t:main-theorem-intro}, we have
\begin{equation}
    \gap(P_{\mathrm{st}}) \geq \hat D \min\set*{c_2,\, c_1 \underline\epsilon^7} \underline\epsilon^5\,,
\end{equation}
where
\begin{equation}\label{e:D-def-intro}
    \hat D(d, \epsilon_0) = \frac{1}{8}c_d \exp\paren*{-\bar C_\mathrm{BV} - \paren*{3 + \frac{2}{\epsilon_0}}\norm{U}_{L^\infty}} h_0^2\,.
\end{equation}
\end{corollary}
We omit the proof for this corollary, as it follows by combining the estimates developed for the parallel tempering chain with standard arguments for simulated tempering (see, e.g.,~\cite{Zheng03, WoodardSchmidlerEA09}). Since this argument introduces no essentially new ideas beyond those already used in the proof of Theorem~\ref{t:main-theorem-intro}, we do not pursue the details here.
\smallskip

We now briefly discuss some notable features of our results, as well as the main challenges that arise in their proof. As noted earlier, when the potential~$U$ exhibits multiple wells, transitions between modes become increasingly rare at low temperatures. In particular, the spectral gap of the Langevin diffusion is known to be exponentially small in the temperature~$\epsilon$, i.e., of order $O(\exp\paren*{-1/\epsilon})$ (see, e.g.,~\cite{BovierGayrardEA05, Kolokoltsov00, Pavliotis14, Arr67}). Consequently, the mixing time grows exponentially as~$\epsilon \to 0$, making efficient sampling prohibitively difficult in this regime.

Our main result (Theorem~\ref{t:main-theorem-intro}) shows that parallel tempering fundamentally alters this picture. For the same class of multi-well potentials, the resulting chain admits a spectral gap that is polynomial in~$\epsilon$ (of order~$11$), representing an exponential improvement over the classical behavior. This provides a theoretical explanation for the empirical success of tempering-based methods in overcoming metastability.

We make a few remarks on some notable features of our results. First, the algorithm does not require any prior structural information about the potential~$U$, such as the locations or depths of its wells, nor any explicit decomposition of the state space. The method only assumes access to evaluations of~$U$, as needed to implement the Metropolis updates and swap/temperature update moves. Second, while we present the result in the setting of a symmetric double-well potential for clarity, the argument extends naturally to more general landscapes with multiple wells and non-degenerate saddles of arbitrary indices. In fact, our main theorem (Theorem~\ref{t:main-theorem}) is proved under the more flexible Assumption~\ref{a:massRatioBound}, which allows for wells of unequal depth, provided that each carries a non-negligible fraction of the total mass.

Another important aspect is that we work directly with explicit, time-discretized Markov chains rather than idealized continuous-time diffusions. In particular, we consider Metropolis-type dynamics that exactly preserve the Gibbs distribution at each temperature level. This reflects practical implementations and avoids relying on assumptions about exact sampling from Langevin diffusions; indeed, while ergodicity of the continuous-time dynamics is relatively well understood under suitable conditions, the ergodic properties of their time-discretized counterparts (such as MALA, MALTA, or Metropolis random walk) for general potentials are already highly nontrivial (see, e.g.,~\cite{MattinglyStuartEA02, BouRabeeHairer13, BouRabeeVandenEijnden10, DurmusMoulines17}).

At a high level, the proof of Theorem~\ref{t:main-theorem} combines a decomposition of the state space with an analysis of the spectral gaps of the corresponding restricted chains. Intuitively, the highest-temperature chain facilitates movement between wells, while lower-temperature chains ensure rapid mixing within each well. Turning this intuition into a rigorous argument, however, presents several challenges. In particular, near saddle points, the local geometry of the potential is unfavorable, in the sense that the Laplacian of~$U$ may become positive, and the dynamics may not exhibit a clear tendency to move toward lower energy. Moreover, the behavior of the chain near the boundaries of the decomposition is delicate, as proposed moves may be rejected and the effective dynamics depend subtly on both the geometry of the boundary and the acceptance mechanism.

We address these issues through a careful perturbation of the potential and a corresponding control of the restricted dynamics, which together ensure that the chain is still driven toward local minima despite these obstacles. A more detailed discussion of this key step is given in Section~\ref{s:lyapunov}.

\subsection{Motivation and Background}\label{s:motivation-and-background}
Sampling from the Gibbs distribution is a central problem in a variety of fields, including statistical physics, statistics, and theoretical computer science (see, e.g.,~\cite{Krauth06, LevinPeres17, RobertCasella99, GelmanCarlinEA14}). In many applications, one is faced with multimodal distributions arising, for instance, from phase transitions in statistical mechanics models or from complex posterior landscapes in Bayesian inference (see, e.g.,~\cite{BovierdenHollander15, BrooksGelman11}). In such settings, naive sampling methods mix prohibitively slowly due to energy barriers separating the modes. This metastable behavior presents a fundamental challenge for Markov Chain Monte Carlo (MCMC) algorithms.

To address this difficulty, a number of advanced sampling methods have been developed, including annealed importance sampling~\cite{Neal01}, sequential Monte Carlo~\cite{DelMoralDoucetEA06}, parallel tempering~\cite{SwendsenWang86}, and simulated tempering~\cite{MarinariParisi92}. These methods exploit a sequence of temperatures: at high temperatures, the distribution is flattened and global mixing is facilitated, while at lower temperatures, the chain mixes efficiently within local regions near the modes.

There has been substantial progress in providing rigorous guarantees for the efficiency of such methods. A notable line of work, exemplified by~\cite{WoodardSchmidlerEA09} and building on the Markov chain decomposition framework of~\cite{MadrasRandall02}, shows that tempering-based algorithms mix rapidly provided that the state space can be decomposed into regions with fast local mixing, together with sufficient overlap between neighboring temperature levels. In particular, they establish polynomial spectral gap bounds in settings where the target distribution is a mixture of Gaussian components. Subsequent works~\cite{LeeRisteskiEA18, GeLeeEA20} use similar ideas based on Markov chain decomposition to treat mixtures of log-concave distributions, again obtaining polynomial mixing guarantees for simulated tempering. In a different direction, \cite{HanIyerEA26} prove that annealed sequential Monte Carlo methods targeting multimodal Gibbs distributions achieve polynomial computational complexity under suitable structural assumptions on the potential.

Despite these advances, most existing results rely on relatively explicit structural assumptions on the target distribution, such as mixture representations or log-concavity within each mode. In contrast, for general Gibbs distributions arising from potentials with multiple wells, much less is known about the quantitative convergence of tempering-based algorithms.

The goal of this work is to address this gap. We consider parallel tempering (and, by extension, simulated tempering) for multimodal Gibbs distributions, under assumptions that ensure a well-behaved multi-well structure but do not require prior knowledge of the locations or shapes of the wells. Our results provide rigorous polynomial spectral gap bounds in this setting, thereby extending the scope of previous analyses beyond mixture-based models.

\subsection{Plan of the paper}

In Section~\ref{s:main-result}, we state the precise assumptions (Section~\ref{s:assumptions}) and the main theorem (Section~\ref{ss:main-result}). We also list the key intermediate lemmas (Lemmas~\ref{l:lyapunov}--\ref{l:large-temperature-gap}) that will be used in its proof (Section~\ref{s:key-lemmas}).

In Section~\ref{s:lyapunov}, we prove Lemma~\ref{l:lyapunov}. In Section~\ref{s:hatp-construct}, we introduce the perturbation and establish its properties. In Section~\ref{s:lyapunov-estimates}, we derive the estimates needed to verify the Lyapunov condition and complete the proof of Lemma~\ref{l:lyapunov}. The proofs of these estimates are deferred to Section~\ref{s:lyapunov-estimate-proof}, while the properties of the perturbation are proved in Section~\ref{s:Phat-properties-proof}.

In Section~\ref{s:sg-restricted}, we prove Lemmas~\ref{l:sg-restricted} and~\ref{l:large-temperature-gap}. Section~\ref{ss:sg-restricted-proof} introduces auxiliary lemmas and uses them to establish Lemma~\ref{l:sg-restricted}. Section~\ref{s:large-temperature-gap-proof} is devoted to the proof of Lemma~\ref{l:large-temperature-gap}. The auxiliary results stated without proof in Section~\ref{ss:sg-restricted-proof} are proved in Section~\ref{s:restricted-to-minimum--normal-gap-proof}.

Finally, in Section~\ref{s:overlap}, we collect additional estimates and combine them with the preceding results to complete the proof of Theorem~\ref{t:main-theorem}.

\subsection*{Acknowledgement} The author thanks Gautam Iyer and Alan Frieze for their helpful suggestions and advice.

\section{Main result and the key lemmas}\label{s:main-result}
In this section, we state the assumptions and the main result of the paper. 
The assumptions are given in Section~\ref{s:assumptions}, the main result in Section~\ref{ss:main-result}, and the key lemmas required for the proofs are listed in Section~\ref{s:key-lemmas}. These lemmas will be proved in the subsequent sections.

\subsection{Assumptions}\label{s:assumptions}

We begin by assuming that~$U$ is a sufficiently regular double-well potential with nondegenerate critical points.

\begin{assumption}\label{a:criticalpts}
The function~$U \in C^6(\mathbb T^d,\mathbb R)$ has nondegenerate Hessian at all critical points and exactly two local minima, located at~$m_1$ and~$m_2$. 
We normalize~$U$ so that
\begin{equation}\label{e:Upositive}
    0 = U(m_1) \leq U(m_2)\,.
\end{equation}
\end{assumption}

Our next assumption concerns the saddle point separating the two minima. 
Define the saddle height between~$m_1$ and~$m_2$ as the minimal energy barrier required to transition between them:
\begin{equation}\label{e:UHatDef}
    \bar{U} = \bar{U}(m_1,m_2) 
    \defeq 
    \inf_{\omega} \sup_{t\in [0,1]} U(\omega(t))\,,
\end{equation}
where the infimum is taken over all continuous paths~$\omega \in C([0,1]; \mathbb T^d)$ such that~$\omega(0)=m_1$ and~$\omega(1)=m_2$.

\begin{assumption}\label{a: nondegeneracy}
The saddle height~$\bar U$ between~$m_1$ and~$m_2$ is attained at a unique critical point~$x=0$ of Morse index one. 
Equivalently, the Hessian~$D^2U(0)$ has eigenvalues
\begin{equation}
    -\lambda_u < 0 < \lambda_1 \leq \cdots \leq \lambda_{d-1}\,.
\end{equation}
\end{assumption}

Finally, we impose a uniform multimodality condition ensuring that both wells carry non-negligible mass in the temperature range of interest.

Let~$\Omega_i$ denote the basin of attraction of~$m_i$, defined as the set of points whose gradient flow converges to~$m_i$, i.e.,
\begin{equation}\label{e:basin-def}
    \Omega_i
    \defeq
    \set[\Big]{
        y \in \mathbb T^d \st
        \lim_{t\to\infty} y_t = m_i,\ 
        \dot{y}_t = -D U(y_t),\ y_0 = y
    }.
\end{equation}

\begin{assumption}\label{a:massRatioBound}
There exist constants~$0 \leq \underline \theta < \overline \theta \leq \infty$ and~$C_m > 0$ such that
\begin{equation}\label{e:massRatioBound}
    \inf_{\epsilon \in [\underline\theta, \overline\theta]}
    \pi_\epsilon(\Omega_i)
    \geq \frac{1}{C_m^2}\,,
    \quad i = 1,2\,.
\end{equation}
\end{assumption}

In particular, \cite[Lemma 4.4]{HanIyerEA26} shows that if~$U(m_2)=0$ then Assumption~\ref{a:massRatioBound} holds with the constant~$C_m$ only depending on~$U$.

\subsection{Main result}\label{ss:main-result}

In this subsection, we state the main result. To this end, we first define the parallel tempering chain and the Metropolis random walk.

\begin{definition}[Parallel tempering chain]\label{d:pt-def}
Let~$\mathcal X$ be a Polish space and let $(R_k, \mu_k)_{k=0}^N$ be a collection of reversible Markov chains~$R_k$ on~$\mathcal X$, each with stationary density~$\mu_k$. Define transition kernels~$R$ and~$S$ on the product space~$\mathcal X_\pt = \mathcal X^{N+1}$ by
\begin{align}
    R(x,dy) &= \frac{1}{2(N+1)}\sum_{k=0}^N R_k(x_k, dy_k)\,\delta_{x_{[-k]}}\paren*{dy_{[-k]}}\,,\\
    S(x,dy) &= \frac{1}{2N}\sum_{k=0}^{N-1} \delta_{(k,k+1)x}(dy)\,\tau(x,(k,k+1)x) \\
    &\quad + \delta_x(dy)\paren*{1 - \frac{1}{2N}\sum_{k=0}^{N-1}\tau(x,(k,k+1)x)}\,.
\end{align}
Here,
\begin{gather}
    x = (x_0, x_1, \ldots, x_N), \quad 
    x_{[-k]} = (x_0, \ldots, x_{k-1}, x_{k+1}, \ldots, x_N),\\
    (k,k+1)x = (x_0, \ldots, x_{k-1}, x_{k+1}, x_k, x_{k+2}, \ldots, x_N),
\end{gather}
and~$\tau\colon \mathcal X_\pt \times \mathcal X_\pt \to [0,1]$ is the Metropolis acceptance probability
\begin{equation}
    \tau(x,y) = \min\set*{1, \frac{\mu(y)}{\mu(x)}}, 
    \quad \text{where} \quad 
    \mu = \prod_{k=0}^N \mu_k\,.
\end{equation}
The parallel tempering kernel~$P_\pt$ is defined by
\begin{equation}
    P_\pt = SRS\,.
\end{equation}
\end{definition}

It is straightforward to verify that the reversibility of each~$R_k$ implies that~$R$ is reversible with respect to~$\mu$. Moreover, the Metropolis filter~$\tau$, together with the uniform choice of adjacent swaps, ensures that~$S$ is also reversible with respect to~$\mu$. Since both~$R$ and~$S$ are lazy, they are non-negative definite. Consequently, the parallel tempering chain~$P_\pt$ is reversible with respect to~$\mu$ and non-negative definite.

\begin{definition}[(Lazy) Metropolis random walk]\label{d:metropolis-def}
Let~$h > 0$, and let~$\pi$ be a probability density on~$\T^d$. The Metropolis random walk with step size~$h$ and stationary density~$\pi$ is the Markov kernel~$P_{h,\pi}$ on~$\mathcal X = \supp(\pi)$ defined by
\begin{equation}
   P_{h,\pi}(x,dy) = \alpha(x,y)\, r_h(x,y)\,dy 
   + \delta_x(dy)\int_{\mathcal X} (1-\alpha(x,z))\, r_h(x,z)\,dz\,,
\end{equation}
where
\begin{equation}
    \alpha(x,y) = \min\set*{1, \frac{\pi(y)}{\pi(x)}}, 
    \quad 
    r_h(x,y) = \frac{\one_{B(x,h)}(y)}{|B(x,h)|}\,.
\end{equation}
The lazy Metropolis random walk with step size~$h$ and stationary density~$\pi$ is defined by
\begin{equation}\label{e:Thpi-def}
    T_{h,\pi} = \frac{1}{2}\paren*{I + P_{h,\pi}}\,.
\end{equation}
\end{definition}

Recall that the Gibbs distribution~$\pi_\epsilon$ is defined in~\eqref{e:gibbs-d-def}. Throughout the remainder of the paper, we write~$P_{h,\epsilon}$ and~$T_{h,\epsilon}$ in place of~$P_{h,\pi_\epsilon}$ and~$T_{h,\pi_\epsilon}$, respectively.

\begin{theorem}\label{t:main-theorem}
    Let~$U$ be a potential that satisfies the Assumptions~\ref{a:criticalpts}--\ref{a:massRatioBound}.
    Then, there exist~$\eta, c_1,  c_2, C_\mathrm{BV}, c_d > 0$ such that the following holds. For any~$\bar \nu\in\paren*{0, 1/\bar\theta}$ and any given temperatures~$\underline \epsilon , \bar \epsilon$ such that~$ \underline \theta \leq \underline\epsilon < \bar \epsilon \leq \bar \theta$, choose
    \begin{equation}
        N = \ceil*{\frac{1}{\bar\nu \underline \epsilon}}\,,\quad \epsilon_0=\bar \epsilon\,,\quad \epsilon_N = \underline \epsilon\,,
    \end{equation}
    and let~$\paren*{1/\epsilon_k}_{k=0}^N$ be linearly spaced. 
    For each~$k\in\set{0, \ldots, N}$, set
    \begin{equation}\label{e:sg-local-h-choice}
        h_k = \min\set*{\eta \epsilon_k^2, 1}\,,   
    \end{equation}
    and define~$T_{h_k, \epsilon_k}$ as the lazy Metropolis random walk with step size~$h_k$ and the stationary density~$\pi_{\epsilon_k}$. 
    If we define~$P_\pt$ to be the parallel tempering chain as in Definition~\ref{d:pt-def} with the sequence~$(T_{h_k, \epsilon_k}, \pi_{\epsilon_k})_{k=0}^N$, then 
    \begin{equation}\label{e:main-theorem-swc-gap}
        \gap(P_\pt) \geq D \min\set*{c_2,  c_1 \underline\epsilon^7 } \underline\epsilon^4 \,,
    \end{equation}
    where
    \begin{equation}
        D(d,\bar\nu,\epsilon_0) \defeq c_d\bar\nu^4 \exp\paren*{-C_\mathrm{BV} C_m^2 - 2\paren*{ \bar\nu +\frac{1}{\epsilon_0}} \norm{U}_{L^\infty} } h_0^2\,.
    \end{equation}
\end{theorem}

\subsection{Key lemmas}\label{s:key-lemmas}
In this subsection, we introduce the key lemmas that will be used to prove Theorem~\ref{t:main-theorem}. 
We decompose the state space~$\T^d$ into the basins of attraction~$\Omega_1$ and~$\Omega_2$, and estimate the spectral gap of the chain restricted to each basin (Lemmas~\ref{l:sg-restricted}--\ref{l:large-temperature-gap}) via the Lyapunov drift condition in Lemma~\ref{l:lyapunov}. We then show in Section~\ref{s:overlap} that the random walk at the highest temperature level~$\epsilon_0$ has a spectral gap on the entire space~$\T^d$, and that the tempering sequence of stationary measures has sufficient overlap. Finally, we apply the result of~\cite{WoodardSchmidlerEA09} to combine these estimates and conclude that the parallel tempering chain~$P_\pt$ has a spectral gap that is polynomial in the final temperature~$\underline \epsilon$.

To formalize this approach, we first define the restriction of a Markov chain. Intuitively, given a subset~$A$ of~$\mathcal X$ and~$x\in A$, the restricted chain~$P\vert_A(x,\cdot)$ proceeds by sampling~$X \sim P(x,\cdot)$ and accepting the move if~$X \in A$, and otherwise rejecting it (i.e., staying at~$x$).

\begin{definition}[Restriction of a Markov chain]\label{d:Markov-chain-restriction}
Let~$P$ be a transition kernel on~$\mathcal X$, reversible with respect to a probability measure~$\mu$, and let~$A \subset \mathcal X$ be measurable. The restriction of~$P$ to~$A$ is the Markov kernel on~$A$ defined by
\begin{equation}
    P\vert_A (x,dy) = P(x,dy) + \delta_x(dy)\, P(x, \mathcal X \setminus A)\,, \quad x \in A.
\end{equation}
\end{definition}

It is straightforward to verify that~$P\vert_A$ is reversible with respect to the conditioned measure~$\mu\vert_A$, where~$\mu\vert_A = \mu(\cdot \cap A)/\mu(A)$. Moreover, if~$\pi(x) > 0$ for all~$x \in A$, then the restriction of a Metropolis random walk with stationary distribution~$\pi$ to~$A$ coincides with the Metropolis random walk with stationary distribution~$\pi\vert_A$.

\smallskip

The first key result is a Lyapunov drift condition for the restricted chain. This will play a crucial role in deducing lower bounds on the spectral gap of the restricted chain. Establishing this estimate constitutes the main technical component of the paper. We prove it in Section~\ref{s:lyapunov}.

\begin{lemma}[Lyapunov drift for a perturbed potential]\label{l:lyapunov}
Make the Assumptions~\ref{a:criticalpts}--\ref{a: nondegeneracy} on~$U$.
There exist constants~$\hat\epsilon, \eta, \lambda, \gamma, a, b, C_P > 0$ such that for any~$\epsilon \le \hat\epsilon$, we have a perturbed potential $\hat U\colon \T^d \to [0,\infty)$, depending on~$\epsilon$, with the following properties.
\begin{enumerate}
    \item For all~$x\in B(m_1, a \sqrt{\epsilon})$,
    \begin{equation}\label{e:Utilde-equals-U}
        \hat U(x) = U(x)\,.
    \end{equation}
    \item Perturbed potential~$\hat U$ is close to~$U$ in~$L^\infty$ in the sense that
    \begin{equation}\label{e:L^infty-perturb}
        \norm{\hat U - U}_{L^\infty(\T^d)} \leq C_P\epsilon\,.
    \end{equation}
    \item Define $W \colon \T^d \to [1/10,\infty)$ as
\begin{equation}\label{e:Wdef}
W(x) = \exp(\gamma\hat U(x))\,.
\end{equation}
Then, for all $h$ satisfying $h/\epsilon^2 \le \eta$, 
\begin{equation}\label{e:lyapunov}
\hat Q_{h,\epsilon} W(x)
\le
(1-\lambda \gamma h^2) W(x)
+ b \one_{\Omega_1 \setminus B(m_1, a\sqrt{\epsilon})}\,,
\qquad
\forall x \in \Omega_1 .
\end{equation}

Here $\hat Q_{h,\epsilon} = \hat P_{h,\epsilon}\vert_{\Omega_1}$ denotes the restriction of the chain $\hat P_{h,\epsilon}$ to $\Omega_1$, where $\hat P_{h,\epsilon}$ is the Metropolis random walk with step size $h$ and the stationary density $\hat\pi_\epsilon \propto \exp(-\hat U/\epsilon)$.
\end{enumerate}
\end{lemma}

The next two lemmas provide lower bounds on the spectral gap of the Metropolis random walk restricted to a basin, in the regimes of small and large~$\epsilon$, respectively. Their proofs are given in Section~\ref{s:sg-restricted}.

\begin{lemma}[Spectral gap of restricted chain for small~$\epsilon$]\label{l:sg-restricted}
Let~$\hat\epsilon, \eta$ be as in Lemma~\ref{l:lyapunov}. 
Then there exists a constant~$\hat c_1 > 0$, independent of~$\epsilon$ and~$h$, such that for all~$\epsilon \le \hat\epsilon$ and all~$h$ satisfying~$0<h/\epsilon^2 \le \eta$, we have
\begin{equation}\label{e:sg-restricted}
   \gap(Q_{h, \epsilon}) \ge \hat c_1 \frac{h^4}{\epsilon} .
\end{equation}
Here, $Q_{h,\epsilon} = P_{h,\epsilon}\vert_{\Omega_1}$ denotes the restriction of the chain $ P_{h,\epsilon}$ to $\Omega_1$, where $P_{h,\epsilon}$ is the Metropolis random walk with step size $h$ and the stationary density $\pi_\epsilon$.
\end{lemma}

\begin{lemma}[Spectral gap of restricted chain for large~$\epsilon$]\label{l:large-temperature-gap}
Let~$\hat\epsilon, \eta$ be as in Lemma~\ref{l:lyapunov}, and let~$\bar h$ be a constant such that
$\eta \hat\epsilon^2 \leq  \bar h$.
Then, there exists a constant~$c_2(\eta, \hat\epsilon, \bar h) > 0$ such that for all~$\epsilon \ge \hat\epsilon$ and all~$h$ satisfying~$\eta \hat\epsilon^2 \le h \le \bar h$,
\begin{equation}\label{e:large-temperature-gap}
    \gap(Q_{h, \epsilon}) \ge c_2 \,.
\end{equation}
Here, $Q_{h, \epsilon}$ is defined as in Lemma~\ref{l:sg-restricted}.
\end{lemma}

\section{Construction of Lyapunov function (Proof of Lemma~\ref{l:lyapunov})}\label{s:lyapunov}

In this section, we prove Lemma~\ref{l:lyapunov}. Before presenting the proof, we briefly explain the main idea and the main difficulty in constructing the Lyapunov function.

Functions of the form~$\exp(\gamma U)$ or~$U^m$, where~$\gamma > 0$ and~$m\in \mathbb N$, are commonly used as Lyapunov functions for Langevin diffusions and their discretized Markov chains when~$U$ itself is the potential (see, e.g.,~\cite{RobertsTweedie96, RobertsTweedie96a, MattinglyStuartEA02, BouRabeeVandenEijnden10, BouRabeeHairer13}). In particular, for~$\epsilon \leq 1/2$, we observe that
\begin{equation}
    \mathcal L_{\epsilon} e^{U} 
    = \bigl((-1+\epsilon)\abs{DU}^2 + \epsilon \Delta U\bigr)e^U 
    \leq \frac{1}{2}\bigl(\mathcal L_{2\epsilon} U\bigr)e^{U}\,,
\end{equation}
where~$\mathcal L_\epsilon$ is the generator of the overdamped Langevin diffusion at temperature~$\epsilon$,
\begin{equation}
    \mathcal L_\epsilon = -DU \cdot D + \epsilon \Delta\,.
\end{equation}
Suppose for the moment that~$U$ attains a local maximum at the saddle point~$x=0$, or at least that~$\Delta U(0) < 0$. Then there exist constants~$\lambda, C > 0$ such that 
\begin{equation}
\mathcal L_{2\epsilon} U(x) < -\lambda \epsilon\,,\quad \forall x \notin B(m_1, C\sqrt{\epsilon}) \cup B(m_2, C\sqrt{\epsilon})\,.
\end{equation}
Indeed, near the saddle, $\Delta U$ is negative, while away from both the saddle and the local minima, $\abs{DU}$ grows linearly and~$\Delta U$ remains bounded above on the compact state space~$\T^d$. Consequently,
\begin{equation}
    \mathcal L_{\epsilon} e^U 
    \leq -\frac{\lambda}{2}\epsilon e^U\,, 
    \quad \forall x \notin B(m_1, C\sqrt{\epsilon}) \cup B(m_2, C\sqrt{\epsilon}),
\end{equation}
which yields the desired Lyapunov drift condition. From a probabilistic perspective, this reflects the mechanism that near the saddle, the random perturbation provides a small push that helps the particle escape the local maximum, and once it moves away, the drift~$DU$ drives it toward one of the local minima.

However, under Assumption~\ref{a: nondegeneracy}, the condition~$\Delta U(0) < 0$ is not guaranteed, since the Hessian~$D^2U(0)$ may have sufficiently large positive eigenvalues. To address this issue, we introduce a local perturbation~$\hat P$ and define a perturbed potential~$\hat U$ as in~\eqref{e:Utilde-def}. The perturbation is designed so that~$\hat U$ behaves \emph{almost} like a local maximum at the saddle~$x=0$, with eigenvalues~$-\lambda_u, \kappa, \kappa, \ldots, \kappa$ for a small parameter~$\kappa$ (see~\eqref{e:Utilde-almost-maximum-at-saddle}). In particular, this ensures that~$\Delta \hat U(0) < 0$, allowing us to recover the above Lyapunov argument.

We remark that similar perturbation ideas have been used in~\cite[Section 3]{MenzSchlichting14} to establish local spectral gap estimates. However, their construction modifies the basin of attraction, making it difficult to control the behavior of the Lyapunov function near the boundary of the original basin. In contrast, our perturbation is designed to preserve the relevant boundary properties (see Lemma~\ref{l:Utilde-vanish}), which is essential for our analysis.

Finally, an additional difficulty arises from the fact that the Metropolis random walk is restricted to a basin. When the particle is very close to the boundary, one must ensure that there is a non-negligible probability of proposing moves that remain inside the basin, and that the corresponding expected decrease in the potential does not vanish. By the compactness of~$\T^d$ and the regularity of the boundary, the relevant geometric quantities (such as normals and curvature) are uniformly bounded, which allows us to control these probabilities and expectations and compare them to the unrestricted case. This issue will appear naturally in the proof of Lemma~\ref{l:lyapunov-near-saddle-boundary}.

We also emphasize that the magnitude of the perturbation must remain sufficiently small, as in~\eqref{e:Utilde-equals-U}. Indeed, a larger perturbation combined with a direct application of the Holley--Stroock Lemma~\ref{l:holley-stroock} would lead to a spectral gap that is exponentially small in~$\epsilon$, which is too weak for our purposes.

\subsection{Construction of the perturbed potential~$\hat U$}\label{s:hatp-construct}
In this subsection, we construct a small perturbation~$\hat P$ and state the properties required to establish the Lyapunov drift condition for the perturbed potential~$\hat U$, defined in~\eqref{e:Utilde-def}. The proofs of these properties are deferred to the final subsection.

We begin by collecting several well-known regularity properties of the basin of attraction and related objects, which are needed to define the perturbation and state its properties. Throughout the remainder of this section, we drop the index~$1$ from the basin of attraction~$\Omega_1$ and simply write~$\Omega$ for notational brevity.

\begin{lemma}\label{l:facts-on-boundaries}
    Under Assumptions~\ref{a:criticalpts}--\ref{a: nondegeneracy} on~$U$, the following properties hold.
    \begin{enumerate}
        \item\label{i: boundary-regularity} 
        $\partial\Omega$ is a $(d-1)$-dimensional manifold of class~$C^5$.
        
        \item\label{i:pi-well-def} 
        There exists~$r_0 > 0$ such that for any
        \begin{equation}\label{e:Gamma-r0-def}
            x\in \Gamma_{r_0}\defeq \set{y\in \T^d\st \dist(y, \partial\Omega) < r_0}\,,
        \end{equation}
        there exists a unique projection~$\xi(x) \in \partial\Omega$ satisfying
        \begin{equation}\label{e:distance-achieved}
            \dist(x,\partial\Omega) = |x-\xi(x)|\,.
        \end{equation}
        Moreover,
        \begin{equation}\label{e:pi-regularity}
            \xi \in C^4(\Gamma_{r_0}, \partial\Omega)\,.
        \end{equation}

        \item \label{i:tangent-normal-eigenvector}
        For each~$x\in \partial\Omega$, let~$n(x), t_1(x), \ldots, t_{d-1}(x)$ denote an orthonormal collection consisting of the outward unit normal vector and tangent vectors to~$\partial\Omega$ at~$x$, respectively. 
        At the saddle point~$x=0 \in \partial\Omega$, $n(0)$ lies in the (unstable) eigenspace of~$D^2U(0)$ corresponding to the negative eigenvalue~$-\lambda_u$, and each~$t_i(0)$ lies in the (stable) eigenspace corresponding to the positive eigenvalue~$\lambda_i$, for $1\leq i\leq d-1$.
    \end{enumerate}
\end{lemma}

We are now ready to define the perturbation. 
We see from the item~\eqref{i:tangent-normal-eigenvector} in Lemma~\ref{l:facts-on-boundaries} that if we view~$n(x), t_i(x)$ as elements in~$\R^{d\times 1}$, then
\begin{equation}
    H(0) \defeq D^2U(0) = -\lambda_u n(0)n(0)^T + \sum_{i=1}^d \lambda_i t_i(0)t_i(0)^T\,.
\end{equation}
We define
 \begin{equation}\label{e:Ps-def}
    P_s \defeq I - n(0)n(0)^T = \sum_{i=1}^{d-1} t_i(0)t_i^T(0)\,,
\end{equation}
for the projection onto the stable eigenspace of~$H(0)$,
and also denote the stable part of~$H(0)$ as
\begin{equation}\label{e:Hs-def}
H_s \defeq P_s H(0) = H(0)P_s = \sum_{i=1}^{d-1}\lambda_i t_i(0)t_i^T(0)\,.
\end{equation}
Let
\begin{equation}\label{e:kappa-def}
    0< \kappa < \min\set{\bar c_d,1}\lambda_u\,,
\end{equation}
where~$\bar c_d$ is a dimensional constant to be chosen later (see~\eqref{e:cd-def}), and define
\begin{equation}\label{e:K-def}
K \defeq H_s - \kappa P_s .
\end{equation}

\begin{definition}\label{d:P-perturb-def}
Fix~$\chi:[0,\infty) \to [0,1]$ be a smooth cutoff function satisfying
\begin{equation}\label{e:chi-def}
\chi(x) =
\begin{cases}
1 & x \in [0,\tfrac12], \\
0 & x \in [1,\infty),
\end{cases}
\qquad
\chi' \le 0 .
\end{equation}

For any $a, \tilde a, \epsilon > 0$, define $P : \R^d \times \R^d \to [0,\infty)$ by
\begin{equation}\label{e:P-def}
P(y,z)
=
\frac12\, y^T K y\;
\chi\!\left(\frac{|y|^2}{a^2\epsilon}\right)
\chi\!\left(\frac{|z|^2}{\tilde a^2 a^2 \epsilon}\right),
\qquad y,z \in \R^d .
\end{equation}
\end{definition}

Note that the saddle point $x=0$ belongs to $\partial\Omega$, and hence $B(0, r_0) \subset \Gamma_{r_0}$.

\begin{definition}\label{d:Ptilde-def}
    Let~$\chi$ be as in Definition~\ref{d:P-perturb-def} and fix 
    \begin{equation}\label{e:rho-def}
        \tilde a = \paren*{2\frac{\lambda_{d-1}}{\lambda_u}\norm{\chi'}_{L^\infty}}^\frac{1}{2}\quad\text{and}\quad \rho\defeq 2(1+\tilde a)\,.
    \end{equation}
    For any $a,\epsilon>0$ such that $\rho a\sqrt{\epsilon}<r_0/2$, define~$\hat P\colon \T^d \to [0,\infty)$ by
    \begin{equation}\label{e:Ptilde-def}
    \hat P(x) = 
    \begin{cases}
    P(\xi(x), x-\xi(x))
    &\text{ if } x \in B(0,\rho a\sqrt{\epsilon})\,,\\
    0
    &\text{ if } x \in B(0, \rho a\sqrt{\epsilon})^c\,,
    \end{cases}
    \end{equation}
    and define~$\hat U\colon \T^d \to [0,\infty)$ as 
    \begin{equation}\label{e:Utilde-def}
        \hat U \defeq U - \hat P\,.
    \end{equation}
\end{definition}

Since~$\rho a\sqrt{\epsilon} < r_0/2$, item~(\ref{i:pi-well-def}) in Lemma~\ref{l:facts-on-boundaries} implies that~$\xi$ is well-defined on~$B(0, r_0)$, and hence $\hat P$ is well-defined. 
Moreover, the smoothness of~$P$ and the regularity of~$\xi$ in~\eqref{e:pi-regularity} ensure that~$\hat P\in C^4(B(0, \rho a\sqrt{\epsilon}))$. 

We also note that
\begin{equation}
\supp(P)\subset \set{(y,z) \st \abs{y}\leq a\sqrt{\epsilon}\,, \abs{z}\leq \tilde a a\sqrt{\epsilon}}\,,    
\end{equation}
and hence, by the triangle inequality~$\abs{x} \leq \abs{\xi(x)} + \abs{x-\xi(x)}$,
\begin{equation}
    \supp(\hat P) \subset B(0, (1+\tilde a)a\sqrt{\epsilon}) \overset{\eqref{e:rho-def}}{\subset} B(0, \rho a\sqrt{\epsilon})\,,
\end{equation}
which implies that~$\hat P\in C^4(\T^d)$.

Going forward, we assume that~$a,\epsilon$ always satisfy~$\rho a\sqrt{\epsilon} < r_0/2$ so that the perturbation~$\hat P$ is well-defined according to Definition~\ref{d:Ptilde-def}. 
We now state the properties of~$\hat P$ that will be used to prove Lemma~\ref{l:lyapunov} and defer their proofs to the last subsection.
\begin{lemma}\label{l:Utilde-vanish}
For any $x \in \partial\Omega$, 
    \begin{equation}
    \label{e:DPs-normal-vanish}
       D\hat{P}(x)\cdot n(x) = 0\,,\quad n(x)^T D^2 \hat{P}(x) n(x) = 0\,.
    \end{equation}
\end{lemma}

\begin{lemma}\label{l:P-regularity}
There exists a constant~$C$, independent of~$a,\epsilon$, such that for any~$a,\epsilon$ with~$a\sqrt{\epsilon}$ sufficiently small, the perturbation satisfies the global bound
\begin{equation}\label{e:P-regularity}
\norm{D^i\hat{P}}_{L^\infty(\T^d)} \le C\paren*{a\sqrt{\epsilon}}^{2-i}\,, \quad \forall 1\leq i\leq3\,,
\end{equation}
and consequently,
\begin{gather}\label{e:U-regularity}
    \norm{\hat U}_{C^2(\T^d)}
    \leq C\,, \qquad \norm{D^3\hat U}_{L^\infty(\T^d)} \leq C\paren*{a\sqrt\epsilon}^{-1}\,.
\end{gather}
\end{lemma}

\begin{lemma}\label{l:Utilde-saddle-almost-local-maximum}
At the saddle~$x=0$, the perturbation satisfies 
\begin{equation}\label{e:DPtilde-critical-minima}
D\hat{P}(0) = 0\quad\text{and}\quad D^2 \hat{P}(0) = K\,,  
\end{equation}
and consequently,
\begin{equation}\label{e:Utilde-almost-maximum-at-saddle}
    D\hat U(0) = 0\quad\text{and}\quad D^2\hat U(0) = H_u + \kappa P_s\,.
\end{equation}
Here, $P_s$ is defined as in~\eqref{e:Ps-def} and
\begin{equation}\label{e:Hu-def}
H_u = -\lambda_u n(0) n(0)^T
\end{equation}
is the projection of~$H(0)=D^2U(0)$ onto the unstable eigenspace of~$H(0)$.
\end{lemma}

\begin{lemma}\label{l:DU-lb-ind-ae}
There exist constants $r_1, c_0> 0$, independent of~$a,\epsilon$, such that for any~$a, \epsilon$ with~$a\sqrt{\epsilon}$ sufficiently small,
\begin{equation}\label{e:DU-lb}
\abs{D\hat U(x)} \ge c_0|x|\,,\quad \forall x\in B(0, r_1)\,.
\end{equation}
\end{lemma}

\subsection{Estimates required for the Lyapunov condition}\label{s:lyapunov-estimates}
In this subsection, we prove Lemma~\ref{l:lyapunov} as follows. First, in Lemma~\ref{l:lyapunov-to-generator}, we decompose the drift condition into cases depending on whether the initial state is near the saddle point~$x=0$ and whether it is close to the boundary~$\partial\Omega$ of the basin of attraction. Next, in Lemmas~\ref{l:D-almost-symmetric}--\ref{l:zeta_n-positivenesses}, we derive estimates for the terms appearing in the drift condition for each case. Finally, we use these estimates to establish the desired bounds in Lemmas~\ref{l:lyapunov-near-saddle-boundary}--\ref{l:lyapunov-far-critical-inside}, treating each of the four cases separately. Combining these lemmas with the properties of the perturbation~$\hat P$, we obtain the drift condition on the entire basin~$\Omega$, thereby proving Lemma~\ref{l:lyapunov}. For continuity, we defer the proofs of Lemmas~\ref{l:lyapunov-to-generator}--\ref{l:zeta_n-positivenesses} to the next subsection.

Throughout this section, we assume that $a\sqrt{\epsilon}$ is sufficiently small so that the perturbed potential~$\hat U$, defined in~\eqref{e:Utilde-def}, is well-defined and all properties stated in Lemmas~\ref{l:Utilde-vanish}--\ref{l:DU-lb-ind-ae} hold. Eventually, we fix a large $\epsilon$-independent constant~$a$ and choose the threshold~$\hat\epsilon$ sufficiently small so that $a\sqrt{\hat\epsilon}$ is small enough. With this choice, Lemma~\ref{l:lyapunov} holds for all $\epsilon < \hat\epsilon$.

We now introduce notation that will be used throughout the remainder of this section.  For~$h>0$, let the random proposal be given by $Y = x + h\zeta$, where $\zeta \sim \Unif(B(0,1))$, and define the random potential difference
\begin{equation}\label{e:D-def}
D \defeq \hat U(Y) - \hat U(x)\,.
\end{equation}
We also define the event that the proposal exits~$\Omega$ by
\begin{equation}\label{e:E-def}
    E \defeq \set{Y \notin \Omega}\,,
\end{equation}
and set
\begin{equation}
\Omega_h \defeq \set{x\in \Omega \st \dist(x, \partial\Omega) < h}\,.    
\end{equation}
By symmetry,
\begin{equation}\label{e:uniform-mean-variance}
    \E\brak{\zeta} = 0 
    \quad\text{and}\quad 
    \E\brak{\zeta \zeta^T} = \sigma^2 I_d
\end{equation}
for some~$\sigma > 0$. We use~$\sigma$ to denote this variance throughout the section.

Finally, we write $R_1 = O(R_2)$ to mean that
\[
\abs{R_1} \le C \abs{R_2}
\]
for some constant $C$ independent of $a, \nu, \epsilon, h, \gamma$. We also introduce the notation~$$\beta \defeq \epsilon^{-1}$$ for the inverse temperature.
\smallskip

We begin by stating a lemma that decomposes the Lyapunov drift condition into several terms, which will be estimated separately in the subsequent lemmas.

\begin{lemma}\label{l:lyapunov-to-generator}
For any~$a\geq 1$ and $\nu > 0$, there exist~$\hat\epsilon(a), \eta(\nu), \hat  \gamma (\nu)> 0$ such that
for any~$\epsilon, h, \gamma > 0$ such that
\begin{equation}
\epsilon \leq \hat\epsilon\,,\quad
h \le \eta\epsilon^2\,,
\quad\text{and}\quad
\gamma < \hat  \gamma\,,
\end{equation}
$\hat P, \hat U$,$W$, and~$\hat Q_{h, \epsilon}$ are well-defined as in ~\eqref{e:Ptilde-def}, \eqref{e:Utilde-def}, \eqref{e:Wdef}, and Lemma~\ref{l:lyapunov}, respectively, and
\begin{align}\label{e:lyapunov-to-generator-far-from-boundary}
\frac{\hat Q_{h,\epsilon}W}{W}(x) - 1 &\le I_1 + I_2,\quad \forall x \in  \Omega \setminus \Omega_h\,,\\
\label{e:lyapunov-to-generator-close-to-boundary}
\frac{\hat Q_{h,\epsilon}W}{W}(x) - 1 &\le I_1 + I_3,\quad \forall x \in  \Omega_h
\end{align}
where
\begin{align}
I_1 &\defeq \E\!\left[(\exp(\gamma D)-1)\exp(-\beta D)\right]\,, \\
I_2 &\defeq \E\!\left[(\exp(\gamma D)-1)(1-\exp(-\beta D))\one_{\{D<0\}}\right]\,,\\
I_3 &\defeq \E\brak{\paren{\exp\paren{\gamma D} - 1} \paren{ \one_{E^c\cap\set{D \leq 0}}-\exp\paren{-\beta D}\one_{E \cup \set{D \leq 0} }}}\,.
\end{align}
Moreover, the terms $I_1, I_2$, and~$I_3$ satisfy
\begin{align}
\label{e:I1-ub}
I_1 &\le -\gamma \beta \tilde I_1 
      + \frac12 \gamma h^2 \sigma^2 \Delta \hat U(x) 
      + \nu \gamma h^2\,, \quad \forall x \in  \Omega \\
\label{e:I1tilde-ub}
\tilde I_1 &= \E\brak*{D^2}  = h^2\sigma^2 \abs{D\hat U(x)}^2 + O(h^3) \,, \quad\forall x \in \Omega \\
\label{e:I2-ub}
I_2 &\le (1+\nu)\gamma\beta \E\!\left[D^2 \one_{\{D\leq 0\}}\right]\,,
\quad \forall x \in  \Omega\setminus \Omega_h
\\
\label{e:I3-ub}
I_3 &\leq (1+\nu)\gamma\beta \E\brak{D^2 \one_{\set{D\leq 0}  \cup \paren*{\set{D>0}\cap E} }} - \gamma\E\brak{D\one_{E}}\,,
\quad \forall x \in  \Omega_h\,.
\end{align}
\end{lemma}

We notice that~$I_1 \approx \gamma \beta h^2 \sigma^2 \hat{\mathcal L_\epsilon }\hat U$ where
\begin{equation}
    \hat{\mathcal L_\epsilon} = - D\hat U \cdot D + \frac{1}{2}\epsilon\Delta\,.
\end{equation}
All other terms are error terms and we present Lemmas~\ref{l:D-almost-symmetric}--\ref{l:Dexit-equality} to provide further bounds on~\eqref{e:I2-ub} and~\eqref{e:I3-ub}. These results will be used to show that the contribution of the gradient of the perturbed potential~$\hat U$ loses at most a constant factor compared to the generator case, and therefore remains significant when the particle is away from the saddle.

\begin{lemma}\label{l:D-almost-symmetric}
    For any~$a\geq1$, there exists~$\hat\epsilon (a) > 0$ such that for any~$\epsilon$ and~$h$ with~$\epsilon < \hat\epsilon$ and~$h \leq \epsilon^2$, the following holds. If~$x \in \Omega$ satisfies
    \begin{equation}\label{e:DU-lb-D-symmetric}
        \abs{D\hat U(x)} \geq ca\sqrt{\epsilon}\,,
    \end{equation}
    for some~$a,\epsilon, h$-independent constant~$c>0$, then
    \begin{equation}\label{e:D-almost-symmetric}
        \E\brak*{D^2 \one_{\set{D\leq 0}}} = \frac{1}{4} h^2 \sigma^2 \abs{D\hat U(x)}^2 + O(h^{2.75})\,.
    \end{equation}
\end{lemma}

\begin{lemma}\label{l:D^2-exit-positive}
    For any~$a\geq1$, there exist~$\hat\epsilon (a), \eta >0$ such that for any~$\epsilon$ and~$h$ with~$\epsilon < \hat\epsilon$ and~$h \leq \eta \epsilon^2$, the following holds. If~$x \in \Omega_h$ satisfies
    \begin{equation}\label{e:DU-nondegenerate}
        \abs{D\hat U(x)} \geq ca\sqrt{\epsilon}\,,
    \end{equation}
    for some~$a,\epsilon, h$-independent constant~$c>0$, then
    \begin{equation}\label{e:D^2-exit-positive}
        \E\brak*{D^2 \one_{\set{D>0} \cap E}} \leq \frac{1}{4} h^2 \sigma^2 \abs{D\hat U(x)}^2 + O(h^{2.75})\,.
    \end{equation}
\end{lemma}

Finally, we state two lemmas that will be used to ensure that, near the saddle, the Laplacian contribution also loses at most a constant factor compared to the generator case and remains sufficiently strong to decrease the Lyapunov function, despite the possibility that proposed moves exit the basin and are rejected.

\begin{lemma}\label{l:Dexit-equality}
    For any~$a\geq1$, there exists~$\hat\epsilon (a) >0$ such that for any~$\epsilon$ and~$h$ with~$\epsilon < \hat\epsilon$ and~$h \leq \epsilon^2$, the following holds.
    For any~$x\in \Omega_h$ such that $\dist(x,\partial\Omega) = \delta h$ for some~$\delta \in (0,1]$, $\xi(x)$ is well-defined as in Lemma~\ref{l:facts-on-boundaries} and
    \begin{equation}\label{e:Dexit-equality}
        \E\brak{D \one_E} = -h^2 \delta L + \frac{1}{2}h^2 Q + O(h^2\abs{D\hat U(x)}) + O(h^{2.75}) \,,
    \end{equation}
    where
        \begin{align}\label{e:L-def-Dexit-equality}
            L &= n(\xi(x))^T H(\xi(x)) n(\xi(x)) \E\brak{\zeta_n \one_E}\,,\\
            Q &= n^T \hat H(x) n \E\brak*{\zeta_n^2 \one_E} + \sum_i t_i^T \hat H(x) t_i \E\brak*{\zeta_{t,i}^2 \one_E}\,.
        \end{align}
    Here, $H=D^2U$ and $\hat H = D^2 \hat U$ respectively, and~$\zeta_n = \zeta \cdot n(\xi(x))$ and~$\zeta_{t,i} = \zeta \cdot t_i(\xi(x))$.
\end{lemma}

\begin{lemma}\label{l:zeta_n-positivenesses}
There exists a dimensional constant~$w>0$ such that for all sufficiently small~$h$ and~$x\in \Omega_h$, $\xi(x)$ is well-defined as in Lemma~\ref{l:facts-on-boundaries} and
\begin{gather}
    \label{e:zeta_n-positivenesses}
    \E\brak{\zeta_n \one_E} \geq 0\,,\quad \E\brak{\zeta_n^2 \one_{E^c}} \geq w\,.
\end{gather}
Here, $\zeta_n$ is defined as in Lemma~\ref{l:Dexit-equality}.
\end{lemma}

The most delicate part of the proof of the drift condition~\eqref{e:lyapunov} arises when~$x$ is near the saddle and close to the boundary. We therefore begin with this case.

\begin{lemma}\label{l:lyapunov-near-saddle-boundary}
    For any sufficiently large~$a$, there exist $\hat\epsilon (a), \eta, \hat  \gamma,\lambda > 0$ such that for all~$\epsilon, h, \gamma$ with
    \begin{equation}\label{e:small-temperature-small-stepsize}
\epsilon \leq \hat\epsilon\,,\quad
h \le \eta\epsilon^2\,,
\quad\text{and}\quad
\gamma < \hat  \gamma\,,
\end{equation}
    it holds that
    \begin{equation}
        \hat Q_{h, \epsilon}W / W (x) - 1 \leq -\lambda \gamma h^2\,, \forall x \in \Omega_h \cap B(0, \rho a\sqrt{\epsilon})\,. 
    \end{equation}
    Here, $\rho$ is defined as in~\eqref{e:rho-def}.
\end{lemma}

\begin{proof}
Eventually, for a sufficiently large~$a \geq 1$, we will choose a small constant~$\nu \leq 1$ depending on~$a$. 
For the moment, fix~$a$ and~$\nu$. By Lemma~\ref{l:lyapunov-to-generator}, we can find~$\hat\epsilon$, $\eta$, and~$\hat \gamma$ such that for all~$\epsilon$ and~$h$ satisfying~\eqref{e:small-temperature-small-stepsize}, the bounds~\eqref{e:lyapunov-to-generator-close-to-boundary}, \eqref{e:I1-ub}, \eqref{e:I1tilde-ub}, and~\eqref{e:I3-ub} hold.

\textbf{Step 1: Near the saddle, i.e.\ $x \in \Omega_h \cap B(0, ca\sqrt{\epsilon})$ for some small~$c>0$.}

We combine~\eqref{e:lyapunov-to-generator-close-to-boundary}, \eqref{e:I1-ub}, \eqref{e:I3-ub}, and the bound 
\begin{gather}
\beta h \leq \eta \epsilon \leq \eta\,,
\end{gather}
so that after shrinking~$\eta$ if necessary, we obtain
\begin{equation}\label{e:I1'I2'I3'-bound-sum-close-saddle}
\hat Q_{h, \epsilon} W / W (x) -1 \leq I_1' + I_2' + I_3'\,,
\end{equation}
where
\begin{align}
I_1' &= -\gamma\beta \E\brak*{D^2\one_{\set{D > 0} \cap E^c}} \leq 0\,,\\
I_2' &= \nu \gamma \beta  \E\brak{D^2 \one_{\set{D\leq 0}  \cup \paren*{\set{D>0}\cap E}}} 
\overset{\eqref{e:I1tilde-ub}}{\leq} 
\nu \gamma \beta h^2\sigma^2 \abs{D\hat U(x)}^2 + \nu \gamma h^2\,,\\
I_3' &= \frac{1}{2} \gamma h^2 \sigma^2 \Delta \hat U (x) + \nu \gamma h^2 - \gamma \E\brak{D\one_E}\,,
\end{align}
and therefore
\begin{equation}
\hat Q_{h, \epsilon} W / W (x) -1
\leq
\nu \gamma \beta h^2\sigma^2 \abs{D\hat U(x)}^2
+ \frac{1}{2} \gamma h^2 \sigma^2 \Delta \hat U (x)
- \gamma \E\brak{D\one_E}
+ 2\nu \gamma h^2\,.
\end{equation}

By choosing~$\hat \epsilon$ sufficiently small and applying Lemma~\ref{l:Dexit-equality}, together with the identity
\begin{equation}
\Delta \hat U (x)
=
\tr \paren*{\hat H(x)}
=
\tr \paren*{P^T \hat H(x)P}
=
n^T \hat H(x) n
+
\sum_i t_i^T \hat H(x) t_i\,,
\end{equation}
where
\begin{equation}
P =
\begin{bmatrix}
n(\xi(x)) & t_1(\xi(x)) & \ldots & t_{d-1}(\xi(x))
\end{bmatrix}
\end{equation}
is an orthogonal matrix, we obtain
\begin{multline}\label{e:PW/W-near-saddle}
\hat Q_{h, \epsilon} W / W (x) -1
\leq\\
\nu  \gamma \beta h^2 \sigma^2\abs{D\hat U(x)}^2
+ \gamma h^2 \delta L
+ \frac{1}{2}\gamma h^2 Q'(x) + C\gamma h^2\abs{D\hat U(x)}
+ 3\nu\gamma h^2\,,
\end{multline}
where~$L$ and~$\delta$ are defined as in Lemma~\ref{e:Dexit-equality} and
\begin{gather}
Q'(x) = n^T \hat H(x) n  w_n(x) + \sum_i t_i^T \hat H(x) t_i w_{t,i}(x)\,,\\
w_n(x) = \E\brak*{\zeta_n^2 \one_{E^c}}\,, \quad
w_{t, i}(x) = \E\brak*{\zeta_{t,i}^2 \one_{E^c}}\,.
\end{gather}

We observe that at $x=0$, the saddle point,
\begin{gather}
n(\xi(x))^T \hat H(x) n(\xi(x))
=
n(0)^T \hat H(0) n(0)
\overset{\eqref{e:Utilde-almost-maximum-at-saddle}}{=}
-\lambda_u\,,\\
t_i^T (\xi(x)) \hat H(x) t_i(\xi(x))
=
t_i^T (0) \hat H(0) t_i(0)
\overset{\eqref{e:Utilde-almost-maximum-at-saddle}}{=}
\kappa\,.
\end{gather}

We observe that the maps
\begin{gather}
g_n\colon B(0, \rho a\sqrt{\epsilon}) \to \R\,, 
\quad x \mapsto n^T(\xi(x))\hat H(x) n(\xi(x))\,,\\
g_{t,i}\colon B(0, \rho a\sqrt{\epsilon}) \to \R\,, 
\quad x \mapsto t_i^T(\xi(x))\hat H(x) t_i(\xi(x))
\end{gather}
have Lipschitz norm of order $O\paren*{\frac{1}{a\sqrt{\epsilon}}}$. 
Indeed, items~(\ref{i: boundary-regularity}) and~(\ref{i:pi-well-def}) in Lemma~\ref{l:facts-on-boundaries} imply that $n, t_i \in C^4(\partial\Omega)$, and~$\xi \in C^4(\Gamma_{r_0})$ so they have $O(1)$-Lipschitz norm on $B(0,1)$, while $\hat H$ has $O\paren*{\frac{1}{a\sqrt{\epsilon}}}$-Lipschitz norm due to~\eqref{e:U-regularity}. 

We now define
\begin{align}
Q''(x) &=
n(0)^T \hat H(0) n(0) w_n(x)
+
\sum_i t_i(0)^T \hat H(0) t_i(0) w_{t,i}(x)\\
&= -\lambda_u  w_n(x) + \kappa \sum_i w_{t,i}(x)  
\,.
\end{align}

Using $w_n \overset{\eqref{e:zeta_n-positivenesses}}{\geq} w$, we see that for some $a,\epsilon$-independent constant $c>0$ and for all $x\in \Omega_h \cap B(0, ca\sqrt{\epsilon})$,
\begin{align}
Q'(x)
&\leq
\abs{Q'(x) - Q''(x)} + Q''(x) \\
&\leq
\paren*{\abs*{g_n}_{C^{0,1}} + \sum_i\abs{g_{t, i}}_{C^{0,1}}}\abs{x}
-\lambda_u w + (d-1)\kappa \\
\label{e:Q'-away-from-zero}
&\leq -\frac{1}{4}\lambda_u w \,,
\end{align}
provided that~$\kappa$ satisfies
\begin{equation}\label{e:smallkappa-negative-laplacian}
-\lambda_u w + (d-1)\kappa \leq -\frac{\lambda_u w}{2}
\iff
\kappa \leq \frac{\lambda_u w}{2(d-1)}\,,
\end{equation}
which we assumed in~\eqref{e:kappa-def} with the choice of
\begin{equation}\label{e:cd-def}
   \bar c_d = \frac{ w}{2(d-1)}
\end{equation}

Similarly, let~$r_0$ be as in Definition~\ref{d:P-perturb-def} and define $g_L\colon B(0, r_0)\to \R$ as
\begin{equation}
g_L(x) = n(\xi(x))^T H(\xi(x)) n(\xi(x))\,.
\end{equation}
Then~$g_L$ has~$O(1)$-Lipschitz norm, so that for sufficiently small~$a\sqrt{\epsilon}$ and~$x\in B(0, \rho a\sqrt{\epsilon})$,
\begin{equation}
g_L(x)
\leq
g_L(0) + C\abs{x}
\leq
-\lambda_u + Ca\sqrt{\epsilon}
\leq
-\frac{1}{2}\lambda_u < 0\,,
\end{equation}
and hence, combining this with~\eqref{e:zeta_n-positivenesses} implies
\begin{equation}\label{e:L-negative-close-saddle}
L(x) = g_L(x)\E\brak{\zeta_n \one_E} < 0\,,\quad
\forall x\in B(0, \rho a\sqrt{\epsilon})\,.
\end{equation}

Combining~\eqref{e:PW/W-near-saddle}, \eqref{e:Q'-away-from-zero}, \eqref{e:L-negative-close-saddle}, and using the fact that
\begin{equation}\label{e:DU-ub}
\abs{D\hat U(x)}
\overset{\eqref{e:Utilde-almost-maximum-at-saddle}}{=}
\abs{D\hat U(x) - D\hat U(0)}
\leq
M\abs{x}\,, \quad\text{where}\quad M\defeq \norm{D^2\hat U}_{L^\infty}\overset{\eqref{e:U-regularity}}{=}O(1)\,,
\end{equation}
we obtain that for all $x\in \Omega_h \cap B(0, ca\sqrt{\epsilon})$,
\begin{equation}
\hat Q_{h, \epsilon} W / W(x) - 1
\leq
\gamma h^2 \paren*{
\sigma^2 M^2 c^2 a^2\nu
-\frac{1}{8}\lambda_u w
+ cCMa\sqrt{\epsilon} + 3\nu
}\,.
\end{equation}
Shrinking~$\hat\epsilon$ and~$\nu$ if necessary completes this step.

\textbf{Step 2: Away from the saddle but still inside the perturbation region.}

The case $x\in\Omega_h \cap B(0, ca\sqrt{\epsilon})$ has already been treated, so it remains to consider
$x\in \Omega_h \cap B(0, \rho a\sqrt{\epsilon}) \cap B(0, ca\sqrt{\epsilon})^c$.

We first note that~\eqref{e:DU-lb} implies that there exists a constant~$c_2>0$ such that for all sufficiently small~$\epsilon$ and~$x\in B(0, \rho a\sqrt{\epsilon}) \cap B(0, ca\sqrt{\epsilon})^c$,
\begin{equation}\label{e:DU-lb-with-a}
\abs{D\hat U (x)} \geq c_2 a\sqrt{\epsilon}\,.
\end{equation}

This implies that the assumptions of Lemma~\ref{l:D-almost-symmetric} and Lemma~\ref{l:D^2-exit-positive} are satisfied. 
Combining~\eqref{e:lyapunov-to-generator-close-to-boundary}, 
\eqref{e:I1-ub}, \eqref{e:I3-ub}, 
\eqref{e:D-almost-symmetric}, \eqref{e:D^2-exit-positive}, 
\eqref{e:L-negative-close-saddle}, and~\eqref{e:P-regularity} (with~$i=2$), 
and applying Young's inequality
\begin{equation}
2\gamma h^2 \abs{D\hat  U(x)}
\leq
\nu \gamma \beta h^2 \sigma^2 \abs{D\hat U(x)}^2
+
\nu^{-1}\epsilon \sigma^{-2} \gamma  h^2 
\end{equation}
yield that for sufficiently small~$\nu, \eta, \epsilon$,
\begin{equation}
\hat Q_{h, \epsilon} W / W (x) - 1
\leq
-\frac{1}{8}\gamma\beta h^2\sigma^2 \abs{D\hat U(x)}^2
+
M\gamma h^2\,,
\end{equation}
for some large $a,\epsilon$-independent~$O(1)$ constant~$M$.

Combining this with~\eqref{e:DU-lb-with-a}, we obtain
\begin{equation}
\hat Q_{h, \epsilon} W/W(x) - 1
\leq
\paren*{-\frac{1}{8}c_2^2 a^2 \sigma^2 + M} \gamma h^2\,.
\end{equation}

Finally, enlarging~$a$ if necessary completes the proof.
\end{proof}

The next case we consider is when the particle is still near the saddle but far from the boundary. The proof is essentially identical to that of Lemma~\ref{l:lyapunov-near-saddle-boundary}, except that we no longer need to estimate the terms involving the exit event~$E$, since the process remains away from the boundary in this regime.

\begin{lemma}\label{l:lyapunov-near-saddle-inside}
    For any sufficiently large~$a\geq 1$, there exist $\hat\epsilon (a), \eta, \hat  \gamma,\lambda > 0$ such that for all~$\epsilon, h, \gamma$ with
    \begin{equation}
\epsilon \leq \hat\epsilon\,,\quad
h \le \eta\epsilon^2\,,
\quad\text{and}\quad
\gamma < \hat  \gamma\,,
\end{equation}
    it holds that
    \begin{equation}
        \hat Q_{h, \epsilon}W / W (x) - 1 \leq -\lambda \gamma h^2\,, \forall x \in \paren*{\Omega\setminus\Omega_h} \cap B(0, \rho a\sqrt{\epsilon})\,. 
    \end{equation}
    Here, $\rho$ is defined as in~\eqref{e:rho-def}.
\end{lemma}

\begin{proof}
As in the proof of Lemma~\ref{l:lyapunov-near-saddle-boundary}, given a sufficiently large~$a\geq1$, we will choose a small constant~$\nu\leq 1$, depending on~$a$. 
For the moment, fix~$a$ and~$\nu$. By Lemma~\ref{l:lyapunov-to-generator}, we can find~$\hat\epsilon$, $\eta$, and~$\hat \gamma$ such that for all~$\epsilon$ and~$h$ satisfying~\eqref{e:small-temperature-small-stepsize}, the bounds~\eqref{e:lyapunov-to-generator-far-from-boundary}, \eqref{e:I1-ub}, \eqref{e:I1tilde-ub}, and~\eqref{e:I2-ub} hold.

\textbf{Step 1: Near the saddle, i.e.\ $x \in \paren*{\Omega \setminus \Omega_h}\cap B(0, ca\sqrt{\epsilon})$ for some small~$c>0$.}

We combine~\eqref{e:lyapunov-to-generator-close-to-boundary}, \eqref{e:I1-ub}, \eqref{e:I3-ub}, and the bound 
\begin{gather}
\beta h \leq \eta \epsilon \leq \eta\,,
\end{gather}
so that after shrinking~$\eta$ if necessary, we obtain
\begin{equation}\label{e:I1'I2'I3'-bound-sum-close-saddle}
\hat Q_{h, \epsilon} W / W (x) -1 \leq I_1' + I_2' + I_3'\,,
\end{equation}
where
\begin{align}
I_1' &= -\gamma\beta \E\brak*{D^2\one_{\set{D > 0}}} \leq 0\,,\\
I_2' &= \nu \gamma \beta  \E\brak{D^2 \one_{\set{D\leq 0}  }} 
\overset{\eqref{e:I1tilde-ub}}{\leq} 
\nu \gamma \beta h^2\sigma^2 \abs{D\hat U(x)}^2 + \nu \gamma h^2\,,\\
I_3' &= \frac{1}{2} \gamma h^2 \sigma^2 \Delta \hat U (x) + \nu \gamma h^2 \,,
\end{align}
and therefore
\begin{equation}\label{e:PW/W-near-saddle-inside}
\hat Q_{h, \epsilon} W / W (x) -1
\leq
\nu \gamma \beta h^2\sigma^2 \abs{D\hat U(x)}^2
+ \frac{1}{2} \gamma h^2 \sigma^2 \Delta \hat U (x)
+ 2\nu \gamma h^2\,.
\end{equation}
Recall that~\eqref{e:Utilde-almost-maximum-at-saddle} implies
\begin{equation}
    \Delta \hat U (0) = \mathrm{Tr}\paren*{D^2 \hat U(0)} = -\lambda_u + (d-1)\kappa \overset{\eqref{e:smallkappa-negative-laplacian}, w\leq 1}{\leq} -\frac{\lambda_u w}{2}\,,
\end{equation}
and hence, by the regularity of~$\hat U$ in~\eqref{e:U-regularity}, there exists a small~$c>0$, independent of~$a,\epsilon$ such that for any~$x\in B(0, ca\sqrt{\epsilon})$,
\begin{align}
    \Delta \hat U (x)&\leq -\frac{\lambda_u w}{4}\,.
\end{align}
Combining this with~\eqref{e:PW/W-near-saddle-inside} and~\eqref{e:DU-ub} yields that for any~$x\in \paren*{\Omega \setminus \Omega_h} \cap B(0, ca\sqrt{\epsilon})$,
\begin{equation}
\hat Q_{h, \epsilon} W / W(x) - 1
\leq
\gamma h^2 \paren*{
\sigma^2 M^2 c^2 a^2\nu
-\frac{1}{8}\lambda_u w \sigma^2
+ 2\nu
}\,.
\end{equation}

Shrinking~$\nu$ if necessary completes this step.

\textbf{Step 2: Away from the saddle but still inside the perturbation region.}

The case $x\in \paren*{\Omega \setminus \Omega_h} \cap B(0, ca\sqrt{\epsilon})$ has already been treated, so it remains to consider
$x\in \paren*{\Omega \setminus \Omega_h} \cap B(0, \rho a\sqrt{\epsilon}) \cap B(0, ca\sqrt{\epsilon})^c$.

We note that~\eqref{e:DU-lb} implies that there exists a constant~$c_1>0$ such that for all sufficiently small~$\epsilon$ and~$x\in B(0, \rho a\sqrt{\epsilon}) \cap B(0, ca\sqrt{\epsilon})^c$, we have
\begin{equation}
    \abs{D\hat U(x)} \geq c_1a\sqrt{\epsilon}\,.
\end{equation}
This implies that the assumption of Lemma~\ref{l:D-almost-symmetric} is satisfied and hence~\eqref{e:D-almost-symmetric} holds.
Combining~\eqref{e:lyapunov-to-generator-far-from-boundary}, \eqref{e:I1-ub}, \eqref{e:I1tilde-ub}, and~\eqref{e:D-almost-symmetric}, we obtain that for sufficiently small~$\nu, \eta, \epsilon$,
\begin{align}
\hat Q_{h, \epsilon} W / W (x) - 1
&\leq
-\frac{1}{2}\gamma\beta h^2\sigma^2 \abs{D\hat U(x)}^2
+
M\gamma h^2\\
&\leq \paren*{-\frac{1}{2}c_1^2 a^2 \sigma^2 + M} \gamma h^2\,.
\end{align}
for some large $a,\epsilon$-independent~$O(1)$ constant~$M$.
Finally, enlarging~$a$ if necessary completes the proof.
\end{proof}

Outside the perturbation region~$B(0, \rho a\sqrt{\epsilon})$, the argument becomes simpler, as we work directly with the original potential~$U$. When the particle is close to the boundary, however, the exit event must still be taken into account.

\begin{lemma}\label{l:lyapunov-far-saddle-boundary}
    For any sufficiently large~$a\geq 1$, there exist $\hat\epsilon (a), \eta, \hat  \gamma,\lambda > 0$ such that for all~$\epsilon, h, \gamma$ with
    \begin{equation}
\epsilon \leq \hat\epsilon\,,\quad
h \le \eta\epsilon^2\,,
\quad\text{and}\quad
\gamma < \hat  \gamma\,,
\end{equation}
    it holds that
    \begin{equation}
        \hat Q_{h, \epsilon}W / W (x) - 1 \leq -\lambda \gamma h^2\,, \forall x \in \Omega_h \cap B(0, \rho a\sqrt{\epsilon})^c\,. 
    \end{equation}
    Here, $\rho$ is defined as in~\eqref{e:rho-def}.
\end{lemma}

\begin{proof}
    Note that by the definition of~$\hat U$ in Definition~\ref{d:Ptilde-def}, $\hat U = U$ on~$B(0,\rho a \sqrt{\epsilon})^c$. This implies that there exists a constant~$C_U$, independent of~$a,\epsilon$, such that for all~$a,\epsilon$ with~$a\sqrt{\epsilon}$ sufficiently small,
    \begin{equation}\label{e:DtildeU-lb-far-from-saddle}
        \abs{D\hat U(x)} = \abs{DU(x)} \geq C_Ua\sqrt{\epsilon}\,,\quad \forall x\in B(0,\rho a\sqrt{\epsilon})^c\,.
    \end{equation}
    This implies that the assumptions of Lemma~\ref{l:D-almost-symmetric} and Lemma~\ref{l:D^2-exit-positive} are satisfied so that~\eqref{e:D-almost-symmetric} and~\eqref{e:D^2-exit-positive} hold.
    Moreover, using~\eqref{e:Dexit-equality} and the Young's inequality
    \begin{equation}
        Ch^2\abs{D\hat U(x)} \leq  \frac{1}{8}\beta h^2\sigma^2\abs{D\hat U(x)}^2 + 2C^2\sigma^{-2}\epsilon h^2
    \end{equation}
    yields 
    \begin{equation}\label{e:gammaD-exit-ub-far-saddle-boundary}
        \abs*{\gamma \E\brak*{D\one_E}} \leq \frac{1}{8}\gamma\beta h^2\sigma^2\abs{D\hat U(x)}^2 + M\gamma h^2 
    \end{equation}
    for some large~$a,\epsilon$-independent~$O(1)$ constant~$M$ and all sufficiently small~$\eta, \epsilon$.
    Using~\eqref{e:lyapunov-to-generator-close-to-boundary}, \eqref{e:I1-ub}, \eqref{e:I1tilde-ub}, \eqref{e:I3-ub}, \eqref{e:D-almost-symmetric}, \eqref{e:D^2-exit-positive}, \eqref{e:gammaD-exit-ub-far-saddle-boundary}, and shrinking~$\nu, \eta, \epsilon$ if necessary, we obtain
    \begin{equation}
    \hat Q_{h, \epsilon} W / W (x) - 1
    \leq
    -\frac{1}{8}\gamma\beta h^2\sigma^2 \abs{D\hat U(x)}^2
    +
    M\gamma h^2\,,
    \end{equation}
    
    Combining this with~\eqref{e:DtildeU-lb-far-from-saddle}, we obtain
    \begin{equation}
    \hat Q_{h, \epsilon} W/W(x) - 1
    \leq
    \paren*{-\frac{1}{8}C_U^2 a^2 \sigma^2 + M} \gamma h^2\,.
    \end{equation}
    Finally, enlarging~$a$ if necessary completes the proof.
    
\end{proof}

\begin{lemma}\label{l:lyapunov-far-critical-inside}
    For any sufficiently large~$a\geq 1$, there exist $\hat\epsilon (a), \eta, \hat  \gamma,\lambda > 0$ such that for all~$\epsilon, h, \gamma$ with
    \begin{equation}
\epsilon \leq \hat\epsilon\,,\quad
h \le \eta\epsilon^2\,,
\quad\text{and}\quad
\gamma < \hat  \gamma\,,
\end{equation}
    it holds that
    \begin{equation}
        \hat Q_{h, \epsilon}W / W (x) - 1 \leq -\lambda \gamma h^2\,,\quad \forall x \in \paren*{\Omega\setminus\Omega_h} \cap B(0, \rho a\sqrt{\epsilon})^c \cap B(m_1,  a\sqrt{\epsilon})^c\,. 
    \end{equation}
    Here, $\rho$ is defined as in~\eqref{e:rho-def}.
\end{lemma}

\begin{proof}
    Recall that~$\hat U = U$ on~$B(0, \rho a\sqrt{\epsilon})^c$. Therefore, for all sufficiently small~$a\sqrt{\epsilon}$,
\begin{equation}\label{e:DtildeU-lb-far-from-critical}
    \abs{D\hat U(x)} = \abs{DU(x)} \geq C_U a\sqrt{\epsilon}\,,\quad 
    \forall x\in B(0,\rho a\sqrt{\epsilon})^c \cap B(m_1, a\sqrt{\epsilon})^c\,,
\end{equation}
since~$U$, being a Morse function, satisfies
\[
\abs{DU(x)} \geq C_U \min\set*{1, \dist(x, S)}
\]
for some constant~$C_U>0$ and all~$x\in \T^d$, where~$S=\set{m_1, m_2, 0}$ denotes the set of critical points of~$U$.
    Then, Lemma~\ref{l:D-almost-symmetric} applies and~\eqref{e:D-almost-symmetric} holds.
    Using~\eqref{e:lyapunov-to-generator-far-from-boundary}, \eqref{e:I1-ub}, \eqref{e:I1tilde-ub}, \eqref{e:I2-ub}, and~\eqref{e:DtildeU-lb-far-from-critical}, we obtain that for sufficiently small~$\nu, \eta, \epsilon$,
    \begin{align}
    \hat Q_{h, \epsilon} W / W (x) - 1
    &\leq
    -\frac{1}{2}\gamma\beta h^2\sigma^2 \abs{D\hat U(x)}^2
    +
    M\gamma h^2\\
    &\leq \paren*{-\frac{1}{2}C_U^2 a^2 \sigma^2 + M} \gamma h^2\,.
    \end{align}
    for some large $a,\epsilon$-independent~$O(1)$ constant~$M$.
    Finally, enlarging~$a$ if necessary completes the proof.
    \end{proof}

Combining the above lemmas, which treat the four cases separately, we obtain the proof of Lemma~\ref{l:lyapunov}.

\begin{proof}[Proof of Lemma~\ref{l:lyapunov}]
Fix sufficiently large~$a$ such that Lemma~\ref{l:lyapunov-near-saddle-boundary}--\ref{l:lyapunov-far-critical-inside} hold. 
Then, for sufficiently small~$\epsilon$, the perturbation~$\hat P$ is well-defined as in~\eqref{e:Ptilde-def}, and shrinking~$\epsilon$ further if necessary, the definition of~$\hat P$ implies $\hat U = U$ on~$B(m_1, a \sqrt{\epsilon})$. This completes the proof for~\eqref{e:Utilde-equals-U}.
Moreover, applying \eqref{e:P-regularity} with~$i=0$ implies the second property~\eqref{e:L^infty-perturb} of~$\hat U$. 

For the Lyapunov condtion~\eqref{e:lyapunov}, Lemma~\ref{l:lyapunov-near-saddle-boundary}--\ref{l:lyapunov-far-critical-inside} implies that we can find $\hat\epsilon, \eta, \hat  \gamma, \lambda$ such that for all~$\epsilon, h, \gamma>0$ that satisfy~\eqref{e:small-temperature-small-stepsize} and~$x \in \Omega \cap B(m_1, a\sqrt{\epsilon})^c$, it holds that
\begin{equation}\label{e:lyapunov-outside-local-minimum}
        \hat Q_{h, \epsilon}W / W (x) - 1 \leq -\lambda \gamma h^2\,.
\end{equation}
Moreover, \eqref{e:Utilde-equals-U} and the fact that~$U(m_1)=DU(m_1)=0$ imply that for all~$x\in B(m_1, a\sqrt{\epsilon})$,
\begin{equation}
 \hat U(x) = U(x) \leq C_U\abs{x}^2 \leq C_U a^2\epsilon\,,   
\end{equation}
and
\begin{equation}
    \hat U(Y) \overset{\eqref{e:D-def}}{=} \hat U(x) + D \overset{\eqref{e:U-regularity}}{\leq} \hat U(x) + Ch\,.
\end{equation}
Thus, for sufficiently small~$\epsilon, \eta, \gamma$, we obtain that for all~$x \in B(m_1, a\sqrt{\epsilon})$,
\begin{equation}\label{e:lyapunov-inside-minima}
     \hat Q_{h, \epsilon}W(x) \leq \E\brak*{\exp\paren*{\gamma \max\set*{\hat U(Y), \hat U(x)}}} \leq 1 \,.
\end{equation}
Setting~$b=1$ and combining~\eqref{e:lyapunov-outside-local-minimum} with~\eqref{e:lyapunov-inside-minima} completes the proof for the third property~\eqref{e:lyapunov} of~$\hat U$.
\end{proof}

\subsection{Proofs for the Lyapunov estimates}\label{s:lyapunov-estimate-proof}
In this subsection, we provide the proofs of the Lyapunov estimates in Lemmas~\ref{l:lyapunov-to-generator}--\ref{l:zeta_n-positivenesses} that were used in the previous subsection. We begin by proving Lemma~\ref{l:lyapunov-to-generator} using Taylor expansions and the definition of the Metropolis random walk.

\begin{proof}[Proof of Lemma~\ref{l:lyapunov-to-generator}]
Given any~$a\geq 1$, we set~$\hat\epsilon $ such that~$a\sqrt{\hat\epsilon} < r_0$, where~$r_0$ is defined as in Definition~\ref{d:Ptilde-def}. Then, for any~$\epsilon < \hat\epsilon$, $\hat P$ in~\eqref{e:Ptilde-def} is well-defined and so is~$\hat U$ as in~\eqref{e:Utilde-def}. Moreover, given any~$\gamma, h>0$, $W$ and~$\hat Q_{h, \epsilon}$ are also well-defined as in~\eqref{e:Wdef} and Lemma~\ref{l:lyapunov}. 
\smallskip

$\mathbf{\textbf{Far from boundary: } x\in \Omega\setminus \Omega_h}\,.$ 
We first prove~\eqref{e:lyapunov-to-generator-far-from-boundary} and the bounds~\eqref{e:I1-ub} and~\eqref{e:I2-ub}.
Let $\gamma \leq 1 $ and assume~$h \leq \epsilon^2 \leq 1$. Given any~$a \geq 1$, let~$a\sqrt{\epsilon}$ be sufficiently small such that~\eqref{e:P-regularity} implies
\begin{equation}\label{e:Utilde-regularity}
    \norm{\hat U}_{C^2} = O(1)\quad\text{and}\quad \norm{\hat U}_{C^3} = O\paren*{\paren*{a\sqrt{\epsilon}}^{-1}}\,, 
\end{equation}
We define a function~$g: \R \to \R$ such that
\begin{equation}
   g(y) = \paren{e^{\gamma  y} - 1} \min\set{1, e^{-\beta y}}\,.
\end{equation}
Fixing~$x\in\Omega$ such that~$\dist(x, \partial\Omega) > h$, setting~$D$ as in~\eqref{e:D-def}, and using the definition of~$\hat Q_{h, \epsilon}$ imply
\begin{align}
    &\frac{\hat Q_{h,\epsilon}W}{W}(x) - 1 = \E\brak{g(D)} = \E \brak{g(D) \one_{\set{D > 0}} }+ \E \brak{g(D) \one_{\set{D< 0}} }\\
    &= \E\brak{\paren{\exp\paren{\gamma D}-1} \exp\paren{-\beta D}} + \E\brak{\paren{\exp\paren{\gamma D} - 1} \paren{1-\exp\paren{-\beta D}} \one_{\set{D<0}}}\\
    \label{e:PW/W-1-bound}
    &\leq I_1 + I_2\,,
\end{align}
which proves~\eqref{e:lyapunov-to-generator-far-from-boundary}.
We see that~$D = O(h)$ so using Taylor expansion, we obtain
\begin{align}
    \exp\paren{\gamma D} - 1 &= \gamma D + O\paren{\gamma D}^2\,, \\
    \exp\paren{-\beta D} &= 1 - \beta D + O\paren{\beta D}^2\,.
\end{align}
Thus, 
\begin{align}
    \paren{\exp\paren{\gamma D}-1} \exp\paren{-\beta D} &= \gamma D- \gamma\beta D^2 + \gamma h^2 O(\beta^2  h + \gamma + \gamma\beta h + 
    \gamma \beta^2 h^2)\\
    &= \gamma D- \gamma\beta D^2 + \gamma h^2 O(\eta + \gamma + \gamma\eta)
\end{align}
so for sufficiently small~$\eta$ and~$\gamma$, 
\begin{equation}\label{e:gammaDbetaD-ub1}
    \paren{\exp\paren{\gamma D}-1} \exp\paren{-\beta D} \leq \gamma D- \gamma\beta D^2 + \frac{1}{2}\nu \gamma h^2\,.
\end{equation}
Finally, using Taylor expansion for~$D$ and combining it with $D^3\hat U \overset{\eqref{e:Utilde-regularity}}{=} O\paren*{(a\sqrt{\epsilon})^{-1}}$ and~$a\geq1$, we obtain
\begin{align}
\label{e:D-first-order}
    D &= h\ip{\zeta, D\hat U(x)} +  O\paren{h^2}\\
\label{e:D-second-order}
    &= h\ip{\zeta, DU(x)} + \frac{1}{2}h^2 \ip{\zeta, D\hat U^2(x)\zeta} + O(h^{2.75})\,, 
\end{align}
and hence, combining this with~\eqref{e:uniform-mean-variance},
\begin{align}\label{e:ED-equality}
    \E\brak{D} &\overset{\eqref{e:D-second-order}}{=} \frac12 h^2 \sigma^2 \Delta \hat U + O(h^{2.75})\,,\\
    \label{e:ED^2-equality}
    I_1' = \E\brak{D^2} &\overset{\eqref{e:D-first-order}}{=} h^2\sigma^2 \abs{D\hat U(x)}^2 + O(h^3)\,.
\end{align}
Combining~\eqref{e:gammaDbetaD-ub1} and~\eqref{e:ED-equality}, and decreasing~$\eta$ if necessary, we obtain~\eqref{e:I1-ub}. Using~\eqref{e:ED^2-equality} yields~\eqref{e:I1tilde-ub}.

To estimate~$I_2$, we see that
\begin{equation}\label{e:taylor-inequality}
0\leq e^y-1\leq e^y y \quad\text{and}\quad  0\leq 1-e^{-y} \leq y \,,\quad\forall y \geq 0\,,    
\end{equation}
so that for any given~$\nu>0$, we can choose~$\beta h \leq \eta\epsilon \leq \eta$ sufficiently small and use~$\abs{D}\leq Ch$ for some~$C$, independent of~$\eta$, to satisfy
\begin{equation}\label{e:small-exp-betaD}
    \exp\paren*{-\beta D}\one_{\set{D<0}}\leq \exp\paren{\eta C} \one_{\set{D<0}} \leq (1+\nu)\one_{\set{D<0}}\,,
\end{equation}
and consequently,
\begin{equation}\label{e:I1-bound}
    I_2 \leq (1+\nu)\beta\gamma \E\brak{D^2 \one_{\set{D<0}}}\,,
\end{equation}
which proves~\eqref{e:I2-ub}.
\smallskip

$\mathbf{\textbf{Close to boundary: } x\in\Omega_h}\,.$
Now, we prove~\eqref{e:lyapunov-to-generator-close-to-boundary} and the bounds~\eqref{e:I1-ub} and~\eqref{e:I3-ub}. Again, let $\gamma \leq 1 $ and assume~$h \leq \epsilon^2 \leq 1$. Given any~$a \geq 1$, let~$a\sqrt{\epsilon}$ be sufficiently small such that~\eqref{e:Utilde-regularity} holds. Using the definition of the chain~$\hat Q_{h, \epsilon}$ and the event~$E$ in~\eqref{e:E-def}, we obtain
\begin{align}
    \frac{\hat Q_{h, \epsilon} W}{W}(x) - 1 &= \E\brak{g(D)\one_{E^c}} =\E\brak{g(D)\one_{E^c \cap \set{D > 0}}} + \E\brak{g(D)\one_{E^c \cap \set{D \leq 0}}} \\
    \label{e:PW/W-near-the-boundary}
    &= I_1 + I_3\,,
\end{align}
which proves~\eqref{e:lyapunov-to-generator-close-to-boundary}.
Exactly same proof as in the previous case of~$x\in \Omega\setminus \Omega_h$ works for~\eqref{e:I1-ub} in case~$x\in \Omega_h$. It remains to prove~\eqref{e:I3-ub}. 
Define
\begin{equation}
  A = \paren{\exp\paren{\gamma D} - 1} \paren{ \one_{E^c\cap\set{D \leq 0}}-\exp\paren{-\beta D}\one_{E \cup \set{D \leq 0} }}\,,  
\end{equation}
the integrand in~$I_3$.
We consider four cases, depending on whether~$D>0$ or not, and the exit event~$E$ occurs or not, separately to bound the term~$I_3$. We also repeat using the inequalities~\eqref{e:taylor-inequality}.

On the event~$E_1 = E^c \cap \set{D \leq 0}$, 
\begin{equation}
    A = \paren{1-\exp\paren{\gamma D}} \paren{\exp\paren{-\beta D}-1}\,,
\end{equation}
and use~\eqref{e:taylor-inequality} to see that
\begin{equation}\label{e:A1-ub}
    A\one_{E_1} \leq \gamma\beta D^2\exp(-\beta D)\one_{E_1}\,.
\end{equation}

On the event~$E_2 = E \cap \set{D \leq 0}$, 
\begin{equation}
    A = \paren{1-\exp\paren{\gamma D}} \exp\paren{-\beta D}\,,
\end{equation}
and use~\eqref{e:taylor-inequality} to see that
\begin{equation}\label{e:A2-ub}
    A\one_{E_2} \leq \paren*{-\gamma D + \gamma\beta D^2\exp(-\beta D)}\one_{E_2}\,.
\end{equation}

On the event~$E_3 = E^c \cap \set{D>0}$, $A=0$ so that
\begin{equation}\label{e:A3-ub}
    A\one_{E_3} = 0\,.
\end{equation}
On the event~$E_4 = E \cap \set{D>0}$, 
\begin{equation}
    A = -\paren{\exp\paren{\gamma D}-1} \exp\paren{-\beta D}\,,
\end{equation}
and use~\eqref{e:taylor-inequality} and~$\exp(x) -1 \geq x$ for all~$x>0$ to see that
\begin{equation}\label{e:A4-ub}
    A\one_{E_4} \leq \paren*{-\gamma D + \gamma\beta D^2}\one_{E_4}\,.
\end{equation}
Again, for any given~$\nu>0$, we can choose sufficiently small~$\eta$ such that~\eqref{e:small-exp-betaD} holds and hence, adding~\eqref{e:A1-ub}--\eqref{e:A4-ub} and collecting similar terms, we obtain~\eqref{e:I3-ub}. 
\end{proof}

To prove Lemmas~\ref{l:D-almost-symmetric}--\ref{l:Dexit-equality}, we first introduce two auxiliary lemmas. These concern: (i) the approximation of the normal component of~$D\hat U$, (ii) the smallness of the expectation of the tangential component of the proposal conditioned on the exit event, and (iii) the fact that the covariance of the proposal, conditioned on exiting and moving uphill, remains comparable to the original covariance up to a negligible error. For continuity, we defer the proofs of these auxiliary lemmas until after establishing Lemmas~\ref{l:D-almost-symmetric}--\ref{l:Dexit-equality}.

\begin{lemma}\label{l:DU-normal}
  Let~$a\geq1$ and suppose~$a\sqrt{\epsilon}$ is sufficiently small, and~$h\leq \epsilon^2$.
  Then, for any~$x\in \Omega_h$ such that~$\dist(x, \partial\Omega) = \delta h$ for some $\delta \in (0,1]$, $\xi(x)$ is well-defined as in Lemma~\ref{l:facts-on-boundaries} and
\begin{equation}\label{e:DU-normal}
D\hat U(x)\cdot n(\xi(x)) = -h\delta n(\xi(x))^T H(\xi(x)) n(\xi(x)) + O(h^{1.75})\,.
\end{equation}
\end{lemma}

\begin{lemma}\label{l:small-geometric-term}
For all sufficiently small~$h$ and all~$x\in \Omega_h$, $\xi(x)$ is well-defined as in Lemma~\ref{l:facts-on-boundaries}, and
\begin{gather}\label{e:small-geometric-terms}
    \E\brak{\zeta_{t,i} \one_E} = O(h)\,, \quad 
    \E\brak{\zeta_n \zeta_{t,i}\one_E} = O(h)\,,\quad 
    \E\brak*{\zeta_{t,i}\zeta_{t,j} \one_E} = O(h)\,,\quad \forall i\neq j\,,
\end{gather}
where the implicit constants in the~$O(\cdot)$ notation depend only on the dimension.

Moreover, if we let~$a\geq 1$ and suppose~$a\sqrt{\epsilon}$ is sufficiently small and~$h\leq \epsilon^2$, then for any~$x\in \Omega_h$, $\xi(x)$ is well-defined. If, in addition,
\begin{equation}\label{e:DtildeU-lb-exit-positive}
    \abs{D\hat U(x)_t} \geq ca\sqrt{\epsilon}\,,
\end{equation}
for some constant~$c>0$ independent of~$a,\epsilon,h$, then
\begin{gather}\label{e:zeta_t-cross-exit-positive-small}
    \E\brak*{\zeta_{t,i}\zeta_{t,j} \one_{E\cap \set{D>0}}} = O(h^{0.75})\,,\quad \forall i\neq j\,,\\
    \label{e:zeta_t-exit-positive-dev-small}
    \E\brak*{\zeta_{t,i}^2 \one_{E\cap \set{D>0}}} = \frac{1}{4}\sigma^2 + O(h^{0.75})\,,\\
    \label{e:zeta_n-exit-positive-dev-small}
    \E\brak*{\zeta_n^2\one_{E\cap \set{D>0}}} = \frac{1}{4}\sigma^2 + O(h^{0.75})\,.
\end{gather}
Here, $D\hat U(x)_t$ is the component of~$D\hat U(x)$ on the tangent space spanned by the orthonormal vectors~$(t_i(\xi(x)))_{i=1}^{d-1}$.
\end{lemma}

The main idea in the proof of Lemma~\ref{l:D-almost-symmetric} is that the first-order approximation~$D \approx h \ip{\zeta, D\hat U(x)}$ is symmetrically distributed. As a result,
\begin{equation}
\E\brak{D^2 \one_{\set{D<0}}} \approx \frac{1}{2}\E\brak{D^2} \approx \frac{1}{2}h^2\sigma^2\abs{\tilde DU(x)}^2\,.
\end{equation}
We then carefully estimate the corresponding error terms.
\begin{proof}[Proof of Lemma~\ref{l:D-almost-symmetric}]
Let~$a\geq1$ and suppose~$a\sqrt{\epsilon}$ is sufficiently small, and~$h\leq \epsilon^2$. Using Taylor expansaion and~\eqref{e:P-regularity}, we obtain
\begin{align}
     D &= L+Q+O(h^{2.75})\,, \quad \text{where} \quad L\defeq hD\hat U(x)\zeta \,,\quad Q\defeq\frac{1}{2}h^2 \ip{\zeta, D\hat U^2(x)\zeta}\,,
\end{align}
which implies
\begin{align}
    \E\brak{D^2\one_{\set{{D<0}}}} &= \E\brak{\paren{L+Q}^2 \one_{\set{{D<0}}}} + O(h^{3.75})\\
    \label{e:D^2D-to-L^2D}
    &= \E\brak{L^2\one_{\set{D<0}}} + O(h^3)\,.
\end{align}
We note that
\begin{align}
    \E\brak{L^2 \one_{\set{D < 0}}} &= \E\brak{L^2 \one_{\set{D < 0, L \geq 0} }} + \E\brak{L^2 \one_{\set{D < 0, L < 0}}}\\
    \label{e:L^2D-separate}
    &\leq \E\brak{L^2 \one_{\set{D < 0, L \geq 0} }} + \E\brak{L^2 \one_{\set{L < 0}}}\,,
\end{align}
and
\begin{align}
\set{D<0\leq L} \subset \set{0\leq L < -Q-O(h^{2.75})} &\subset \set{\abs{L} \leq \abs{Q} + O(h^{2.75})} \\
&\overset{\eqref{e:DU-lb-D-symmetric}}{\subset} \set{c\sqrt{\epsilon} h \abs{\ip{v , \zeta}} \leq O(h^2)}\,,
\end{align}
where~$v=D\hat U(x)/\abs{D\hat U(x)}$.
Since the distribution of~$\zeta$ is rotation invariant, we have $\ip{v, \zeta} \overset{d}{\sim} \zeta_1$ and combining this with the fact that~$\beta^2 h\leq 1$, we obtain
\begin{equation}\label{e:PD<0<L-estimate}
    \P\brak{D<0\leq L} \leq O(\epsilon^\frac{3}{2})\,.
\end{equation}
Moreover, $\zeta \overset{d}{\sim} -\zeta$ so that~$L\overset{d}{\sim}-L$ and~$LQ\overset{d}{\sim}-LQ$, which implies
\begin{equation}
    \label{e:L^2L-to-D^2}
    \E\brak{L^2\one_{\set{L<0}}} = \frac{1}{2}\E\brak{L^2} = \frac{1}{2}\paren*{\E\brak{D^2} + O(h^{3.75})}\,.
\end{equation}
Combining this with~\eqref{e:D^2D-to-L^2D}, \eqref{e:L^2D-separate}, \eqref{e:PD<0<L-estimate}, and~\eqref{e:L^2L-to-D^2}, and using~$L^2 = O(h^2)$ imply
\begin{equation}
    \E\brak{D^2 \one_{\set{D < 0}}} \leq \frac{1}{2}\E\brak{D^2} + O(\epsilon^\frac{3}{2} h^2) + O(h^3)\,,
\end{equation}
and using~$\E\brak{D^2} = \E\brak{L^2} + O(h^{3.75}) = h^2\sigma^2\abs{D\hat U(x)}^2 + O(h^{3.75})$, we obtain
\begin{equation}
    \E\brak{D^2 \one_{\set{D < 0}}} \leq \frac{1}{2}h^2\sigma^2\abs{DU(x)}^2 + O(\epsilon^\frac{3}{2} h^2)\,.
\end{equation}
\end{proof}
Using Lemmas~\ref{l:DU-normal}--\ref{l:small-geometric-term}, we directly obtain Lemmas~\ref{l:D^2-exit-positive}--\ref{l:Dexit-equality}.

\begin{proof}[Proof of Lemma~\ref{l:D^2-exit-positive}]
Let~$a\geq1$ and suppose~$a\sqrt{\epsilon}$ is sufficiently small, and~$h\leq \epsilon^2$.
Note that $D\hat U (x) \cdot n(\xi(x)) = O(h)$ in~\eqref{e:DU-normal} so that 
\begin{equation}\label{e:DU-tangent-lb}
    \abs{D\hat  U(x)_t}^2 = \abs{D\hat U(x)}^2 - \abs{D\hat U(x) \cdot n(\xi(x))}^2 \overset{\eqref{e:DU-normal}, \eqref{e:DU-nondegenerate}}{\geq} c^2a^2\epsilon - O(h^2) \geq \frac{1}{2}c^2a^2\epsilon\,,     
\end{equation}
for sufficiently small~$\hat\epsilon, \eta$. 
Here, $D\hat U(x)_t$ is the component of~$D\hat U(x)$ on the tangent space spanned by~$t_i(\xi(x))$ i.e.
\begin{equation}
    D\hat U(x)_t = \sum_{i=1}^{d-1} \paren*{D\hat U(x)_{t,i}} t_i(\xi(x))\,,\quad\text{where}\quad D\hat U(x)_{t,i} = \ip{D\hat U(x), t_i(\xi(x))}\,.
\end{equation}
We similarly define~$D\hat U(x)_n = D\hat U(x) \cdot n(\xi(x))$.
Then, using the Taylor expansion~
\begin{equation}
D = hD\hat U(x)\zeta + O(h^2) = h\paren*{D\hat U(x)_n \zeta_n + \sum_i {D\hat U(x)_{t,i}} \zeta_{t,i}} + O(h^2)
\end{equation}
yields 
\begin{align}\label{e:D-exit-positive-equality}
    \E\brak*{D^2 \one_{E \cap \set{D > 0}}} = h^2\paren*{I_1 + I_2 + I_3}  + O(h^3)\,,
\end{align}
where
\begin{align}
    I_1 &=  \abs{D\hat U(x)_n}^2 \E\brak*{\zeta_n^2 \one_{E \cap \set{D > 0}}} + \sum_i \abs{D\hat U(x)_{t,i}}^2 \E\brak*{\zeta_{t,i}^2\one_{E \cap \set{D > 0}}}\,,\\
    I_2 &=  D\hat U(x)_n \sum_i  D\hat U(x)_{t,i} \E\brak*{\zeta_n \zeta_{t,i} \one_{E \cap \set{D > 0}}}\,,\\
    I_3 &=  \sum_{i\neq j}D\hat U(x)_{t,i} D\hat U(x)_{t,j} \E\brak*{\zeta_{t,i} \zeta_{t,j} \one_{E \cap \set{D > 0}}}\,.
\end{align}
Using~\eqref{e:zeta_t-cross-exit-positive-small}--\eqref{e:zeta_n-exit-positive-dev-small}, we obtain
\begin{align}\label{e:D-exit-positive-I1-I3-ub}
    I_1 \leq \frac{1}{4} \sigma^2 \abs{D\hat U(x)}^2 + O(h^{0.75})\,,\quad I_3 \leq O(h^{0.75})\,,
\end{align}
and using the fact that~$D\hat U(x)_n = O(h)$ in~\eqref{e:DU-normal}. we obtain
\begin{equation}\label{e:D-exit-positive-I2-ub}
    I_2 \leq O(h)\,.
\end{equation}
Combining~\eqref{e:D-exit-positive-equality}, \eqref{e:D-exit-positive-I1-I3-ub}, and~\eqref{e:D-exit-positive-I2-ub} yields~\eqref{e:D^2-exit-positive}.
\end{proof}

\begin{proof}[Proof of Lemma~\ref{l:Dexit-equality}]
Let~$a\geq1$ and suppose~$a\sqrt{\epsilon}$ is sufficiently small, and~$h\leq \epsilon^2$. Using the Taylor expansion for $D$ and the estimate~\eqref{e:P-regularity}, we have
\begin{equation}\label{e:Dexit-taylor-expansion}
    \E\brak{D\one_E} = h\E\brak{D\hat U(x)\zeta \one_E} + \frac{1}{2}h^2\E\brak{\zeta^T \hat  H(x)\zeta \one_E} + O\paren*{h^{2.75}}\,,
\end{equation}
and separating $\zeta$ into tangent and normal components, we get
\begin{align}\label{e:DU-zeta-separation}
    \E\brak{D\hat U(x)\zeta \one_E} &= D\hat U(x) \cdot n(\xi(x))  \E\brak{\zeta_n \one_E} + D\hat U (x)_t \cdot \E\brak{\zeta_t \one_E} \,,\\
    \label{e:H-zeta-separation}
    \E\brak{\zeta^T \hat  H(x)\zeta \one_E} &\overset{\eqref{e:small-geometric-terms}}{=} Q + O(h)\,.
\end{align}

Combining~\eqref{e:DU-normal}, \eqref{e:Dexit-taylor-expansion} \eqref{e:DU-zeta-separation}, \eqref{e:small-geometric-terms}, and~\eqref{e:H-zeta-separation}, we conclude that~\eqref{e:Dexit-equality} holds.
\end{proof}

It remains to prove Lemmas~\ref{l:zeta_n-positivenesses}, \ref{l:DU-normal}, and~\ref{l:small-geometric-term}. To this end, we use the following lemma to characterize the exit event~$E$ and the relationship between the projection and the normal vector. The statements in this lemma are standard, so we defer its proof to Appendix~\ref{a:regularities-of-basin-projection}.

\begin{lemma}\label{l:signed-distance-properties}
Make the Assumptions~\ref{a:criticalpts} and~\ref{a: nondegeneracy} on~$U$.
Define a signed distance function~$d:\T^d \to [0,\infty)$ as
\begin{equation}\label{e:signed-d-def}
    d(x) = 
    \begin{cases}
      -\dist(x, \partial\Omega)  & x \in \Omega\,,\\
        \dist(x, \partial\Omega) & x \not\in \Omega\,,\\
    \end{cases}
\end{equation}
and let~$r_0$ be as in the item~(\ref{i:pi-well-def}) of Lemma~\ref{l:facts-on-boundaries} and~$\Gamma_{r_0}$ be as in~\eqref{e:Gamma-r0-def}.
Then, 
\begin{equation}\label{e:signed-d-regularity}
d\in C^5(\Gamma_{r_0}, (-r_0, r_0))\,,
\end{equation}
and for all~$x\in \Gamma_{r_0}$,
\begin{equation}\label{e:Dd-normal-identity}
Dd(x) = n(\xi(x))^T \quad\text{\and}\quad x = \xi(x) + d(x)n(\xi(x))\,,
\end{equation}
where~$n$ is the outward unit normal vector to~$\partial\Omega$, as defined in item~(\ref{i:tangent-normal-eigenvector}) of Lemma~\ref{l:facts-on-boundaries}.
\end{lemma}

The main idea in the proof of Lemma~\ref{l:zeta_n-positivenesses} is that, for a particle near the boundary to exit the basin, it must have a sufficiently large normal component. We exploit the fact that the exit event~$E$ is almost~$\zeta_n$-measurable, in the sense that it can be sandwiched between two~$\zeta_n$-measurable events up to a small error, to deduce this property. 

\begin{proof}[Proof of Lemma~\ref{l:zeta_n-positivenesses}]
Let~$h < r_0/4$. Then by item~(\ref{i:pi-well-def}) in Lemma~\ref{l:facts-on-boundaries}, for all~$x\in \Omega_h$, $\xi(x)$ is well-defined. Moreover, by the definition and the regularity of the signed distance function~$d$ in~\eqref{e:signed-d-def} and~\eqref{e:signed-d-regularity}, we see that 
\begin{equation}\label{e:exit-positive-signed-distance-function}
E = \set{Y \not\in \Omega} = \set{d(Y) \geq 0}\,,
\end{equation}
and observe that for any~$x\in \Omega_h$ such that~$\dist(x,\partial\Omega)=\delta h$ for some~$\delta \in (0,1]$, 
\begin{multline}\label{e:taylor-signed-d}
    d(x+h\zeta) = d(x) + hDd(x) \zeta + R \overset{\eqref{e:Dd-normal-identity}}{=} -\delta h + hn(\xi(x))^T \zeta + R \geq 0 \\
    \iff \zeta_n \geq \delta - \frac{1}{h}R\,,
\end{multline}
for some remainder term~$R$. 
We use the fact that~$R= h^2 \zeta^T D^2 d (x_*) \zeta$ for some $x_*$ and $\norm{D^2 d}_{L^\infty(\Gamma_{3r_0/4})} \leq C$ to deduce~$\norm{R}_{L^\infty} \leq hC $ and
\begin{equation}\label{e:E-almost-zetan-measurable}
    A_1 \defeq \set{\zeta_n \geq \delta + h C} \subset E=\set{Y \not\in\Omega} \subset A_2 \defeq \set{\zeta_n \geq \delta - hC}\,.
\end{equation}
We note that if~$\delta > hC$ then $\set{Y\not\in \Omega} \subset \set{\zeta_n \geq 0}$ so that $E\brak{\zeta_n \one_E} \geq 0$. Thus, we assume~$\delta \leq h C$ and see that
\begin{align}
    \E\brak*{\zeta_n \one_E} &= \E\brak*{\zeta_n \one_{E\cap \set{\zeta_n > 0 }}} - \E\brak*{\paren*{-\zeta_n} \one_{E\cap \set{\zeta_n < 0 }}}\\
    &\geq \E\brak*{\zeta_n \one_{\set{\zeta_n \geq \delta + h C}}} - \E\brak*{\paren*{-\zeta_n} \one_{\set{ \delta - hC \leq \zeta_n < 0} }}\\
    &\geq \E\brak*{\zeta_n \one_{\set{\zeta_n \geq 2hC}}} - \E\brak*{\paren*{-\zeta_n} \one_{\set{-hC\leq \zeta_n < 0}}}\\
    &\geq \E\brak*{\zeta_1\one_{\zeta_1 \geq 10^{-10}}}-hC > 0\,,
\end{align}
for sufficiently small~$h>0$, where~$\zeta_1$ is the first coordinate of~$\zeta \sim \Unif(B(0,1))$.

Moreover,
\begin{align}
    \E\brak*{\zeta_n^2\one_{E^c}} \geq \E\brak*{\zeta_n^2 \one_{\set{\zeta_n < \delta - hC}}} \geq \E\brak*{\zeta_n^2 \one_{\set{\zeta_n < -hC}}} \geq \E\brak*{\zeta_1^2 \one_{\set*{\zeta_1 < -10^{-10}}}}\,,
\end{align}
for sufficiently small~$h>0$.
These conclude the proof for~\eqref{e:zeta_n-positivenesses}.
\end{proof}

To prove Lemma~\ref{l:DU-normal}, we use the identities~\eqref{e:Dd-normal-identity} and~\eqref{e:DPs-normal-vanish}, together with a Taylor expansion.

\begin{proof}[Proof of Lemma~\ref{l:DU-normal}]
For a given~$a\geq 1$, if we set~$a\sqrt \epsilon < \min\set{1, r_0/2}$, where~$r_0$ is defined as in item~\ref{i:pi-well-def} in Lemma~\ref{l:facts-on-boundaries}, then 
\begin{equation}
    \epsilon^2\leq \sqrt{\epsilon}\leq a\sqrt{\epsilon} \leq \min\set*{1, \frac{1}{2}r_0}\,,
\end{equation}
which implies that for any~$h\leq \epsilon^2$ and~$x \in \Omega_h$, $\xi(x)$ is well-defined.

For the second assertion, using~\eqref{e:Dd-normal-identity} and Taylor expansion yields
\begin{align}
    D\hat U(x) &= D\hat U(\xi(x)-n(\xi(x))h\delta) \\
    &\overset{\eqref{e:P-regularity}}{=} D\hat U(\xi(x)) - h\delta \hat H(\xi(x))n(\xi(x)) + O\paren*{\frac{1}{a\sqrt{\epsilon}} h^2\delta^2}\\
    &=  D\hat U(\xi(x)) - h\delta  \hat  H(\xi(x))n(\xi(x)) + O\paren*{h^{1.75}}\,.
\end{align}
Multiplying~$n(\xi(x))$ on both sides of the display and using~\eqref{e:DPs-normal-vanish} and the Neumann boundary condition~$DU(\xi(x)) \cdot n(\xi(x)) = 0$, we obtain~\eqref{e:DU-normal}.
\end{proof}

As in the proof of Lemma~\ref{l:zeta_n-positivenesses}, we use the almost~$\zeta_n$-measurability of~$E$, the almost~$\zeta_t$-measurability of the event~$\set{D>0}$, and symmetry properties of the relevant conditional distributions to establish Lemma~\ref{l:small-geometric-term}.

\begin{proof}[Proof of Lemma~\ref{l:small-geometric-term}]
As in the proof of Lemma~\ref{l:zeta_n-positivenesses}, if we set~$h < r_0/4$, then for all~$x\in \Omega_h$, $\xi(x)$ is well-defined and the property~\eqref{e:E-almost-zetan-measurable} follows through. 
The symmetries of~$\mathrm{Law}\paren*{\zeta_{t,i} \st \zeta_n}$ and~$\mathrm{Law}\paren*{\zeta_{t,i}\zeta_{t,j} \st \zeta_n}$ for~$i\neq j$ imply
\begin{gather}\label{e:zero-symmetries}
    \E\brak*{\zeta_{t,i}\one_{A_2}} = 0\,, \quad \E\brak*{\zeta_n \zeta_{t,i}\one_{A_2}} = 0\,,\quad\text{and}\quad \E\brak*{\zeta_{t,i} \zeta_{t,j}\one_{A_2}} = 0\,.
\end{gather} 
Hence,
\begin{equation}
    \abs*{\E\brak*{\zeta_{t,i} \one_E }}= \abs*{   \E\brak*{\zeta_{t,i} \one_{A_2\setminus E}}  } \leq \P\brak*{A_2 \setminus E} \leq \P\brak*{A_2\setminus A_1}  \leq Ch\,.
\end{equation}
Similarly, 
\begin{equation}
    \abs*{\E\brak*{\zeta_n \zeta_{t,i} \one_E }} \leq Ch\quad\text{and}\quad \abs*{\E\brak*{\zeta_{t,i}\zeta_{t,j} \one_E }}\leq Ch
\end{equation}
hold and this completes the proof for~\eqref{e:small-geometric-terms}.

For the second assertion, as in Lemma~\ref{l:DU-normal} and Lemma~\ref{e:zeta_n-positivenesses}, for a given~$a\geq 1$, if we set~$a\sqrt{\epsilon} < \min\set*{1, r_0/4}$ then~$h < r_0/4$ so that for any~$x\in \Omega_h$, $\xi(x)$  is well-defined and the property~\eqref{e:E-almost-zetan-measurable} follows through. Now, we fix~$x\in \Omega_h$ and let~$\dist(x,\partial\Omega) = \delta h$ for some~$\delta \in (0,1]$.
Using the Taylor expansion for~$D$ yields
\begin{multline}
    D =  h D\hat U(x)\zeta + \tilde R > 0 \iff
    D\hat U(x) \zeta = (D\hat U(x)\cdot n) \zeta_n + D\hat U(x)_t \cdot \zeta_t> -\frac{1}{h}\tilde R\,,
\end{multline}
where~$\tilde R = h^2 \zeta^T D^2 \hat U(\tilde x_*) \zeta$ for some~$\tilde x_*$.
Combining this with~$\tilde R \overset{\eqref{e:U-regularity}}{=} O(h^2)$ and~$D\hat U(x) \cdot n(\xi(x)) \overset{\eqref{e:DU-normal}}{=} O(h)$, we obtain
\begin{equation}\label{e:D-almost-zetat-measurable}
    A_3 \defeq \set{D\hat U(x)_t \cdot \zeta_t > Ch} \subset \set{D > 0} \subset A_4 \defeq \set{D\hat U(x)_t \cdot \zeta_t > - Ch}\,.
\end{equation}

The symmetry of~$\mathrm{Law}\paren*{\zeta_t \st \zeta_n}$ implies
\begin{equation}
    \E\brak*{\zeta_{t,i}\zeta_{t,j}\one_{\set{D\hat U(x)_t \cdot \zeta_t > 0}} \big\vert \zeta_n} = \E\brak*{\zeta_{t,i}\zeta_{t,j}\one_{\set{D\hat U(x)_t \cdot \zeta_t < 0}} \big\vert \zeta_n}\,,
\end{equation}
and summing them up and using~$\E\brak*{\zeta_{t, i}\zeta_{t,j}\vert\zeta_n}=0$ yields
\begin{equation}\label{e:zetat-hyperplane-mean-zero}
    \E\brak*{\zeta_{t,i}\zeta_{t,j}\one_{\set{D\hat U(x)_t \cdot \zeta_t > 0}} \big\vert \zeta_n} = 0\,.
\end{equation}

We note that
\begin{align}\label{e:zeta_t-exit-positive-break-down}
    \abs*{\E\brak*{\zeta_{t,i}\zeta_{t,j} \one_{E\cap \set{D>0}}} - \E\brak*{\zeta_{t,i}\zeta_{t,j} \one_{A_2 \cap A_4}}} \leq I_1 + I_2\,,
\end{align}
where
\begin{align}
    I_1 &= \abs*{\E\brak*{\zeta_{t,i}\zeta_{t,j} \one_{E\cap \set{D>0}}} - \E\brak*{\zeta_{t,i}\zeta_{t,j} \one_{A_2 \cap \set{D>0}}}}\,,\\
    I_2 &= \abs*{\E\brak*{\zeta_{t,i}\zeta_{t,j} \one_{A_2\cap \set{D>0}}} - \E\brak*{\zeta_{t,i}\zeta_{t,j} \one_{A_2 \cap A_4}}} \,,
\end{align}
and observe that
\begin{align}\label{e:exit-positive-I1-ub}
    I_1 &\overset{\eqref{e:E-almost-zetan-measurable}}{=} \abs*{\E\brak*{\zeta_{t,i}\zeta_{t,j} \one_{(A_2\setminus E) \cap \set{D>0}}}} \leq \P\brak{A_2 \setminus A_1} = Ch\,,\\
    I_2 &\overset{\eqref{e:D-almost-zetat-measurable}}{\leq} \P\brak{A_4\setminus A_3} = \P\brak*{\abs*{D\hat U(x)_t \cdot \zeta_t} \leq Ch}\,.
\end{align}
In particular, setting~$v = D\hat U(x)_t/\abs{D\hat U(x)_t}$ and using~\eqref{e:DtildeU-lb-exit-positive}, $a\geq 1$, and~$h\leq \epsilon^2$ yield
\begin{equation}\label{e:exit-positive-I2-ub}
    I_2 \leq \P\brak*{\abs*{v\cdot \zeta_t} \leq Ch^{0.75}} \leq Ch^{0.75}\,.
\end{equation}

Moreover, 
\begin{align}
    \one_{A_2}\E\brak*{\zeta_{t,i}\zeta_{t,j}\one_{A_4} \st\zeta_n} &\overset{\eqref{e:zetat-hyperplane-mean-zero}}{=} \one_{A_2}\E\brak*{\zeta_{t,i}\zeta_{t,j}\one_{A_4\setminus\set{ D\hat U(x)_t \cdot \zeta_t > 0}} \st\zeta_n}\,,
\end{align}
which implies
\begin{equation}\label{e:exit-positive-final-ub}
    \abs*{\E\brak*{\zeta_{t,i}\zeta_{t,j} \one_{A_2 \cap A_4}}} \leq\P\brak{A_4\setminus A_3} \leq Ch^{0.75}\,.
\end{equation}

Combining~\eqref{e:zeta_t-exit-positive-break-down},\eqref{e:exit-positive-I1-ub}, \eqref{e:exit-positive-I2-ub}, and~\eqref{e:exit-positive-final-ub} yields the bound~\eqref{e:zeta_t-cross-exit-positive-small}.

Similarly, from the symmetry of~$\mathrm{Law}(\zeta_n), \mathrm{Law}(\zeta_t\st \zeta_n)$, we have
\begin{equation}
    \E\brak*{\zeta_n^2 \one_{\set{\zeta_n > 0}}} = \frac{1}{2}\sigma^2 \,,\quad \P\brak*{D\hat U(x)_t \cdot \zeta_t > 0 \st \zeta_n} = \frac{1}{2}\,,
\end{equation}
and hence
\begin{equation}
    \E\brak*{\zeta_n^2 \one_{\set{\zeta_n > 0} \cap \set{D\hat U(x)_t \cdot \zeta_t > 0}}} = \frac{1}{4}\sigma^2\,.
\end{equation}
Combining this with
\begin{gather}
    \abs*{\E\brak*{\zeta_n^2 \one_{\set{\zeta_n > -Ch} \cap \set{D\hat U(x)_t \cdot \zeta_t > 0}}}- \E\brak*{\zeta_n^2 \one_{\set{\zeta_n > 0} \cap \set{D\hat U(x)_t \cdot \zeta_t > 0}}}} \leq Ch^3\,,\\
    \abs*{\E\brak*{\zeta_n^2 \one_{\set{\zeta_n > -Ch} \cap \set{D\hat U(x)_t \cdot \zeta_t > -Ch}}}- \E\brak*{\zeta_n^2 \one_{\set{\zeta_n > -Ch} \cap \set{D\hat U(x)_t \cdot \zeta_t > 0}}}} \leq Ch^{0.75}\,,
\end{gather}
we obtain
\begin{align}
    \E\brak*{\zeta_n^2\one_{E\cap \set{D>0}}} &\leq \E\brak*{\zeta_n^2 \one_{\set{\zeta_n > -Ch} \cap \set{D\hat U(x)_t \cdot \zeta_t > -Ch}}} \leq \frac{1}{4}\sigma^2 + Ch^{0.75}\,.
\end{align}

Similarly, from the symmetry of~$\mathrm{Law}(\zeta_{t,i}\st \zeta_n)$ and~$\mathrm{Law}(\zeta_n, \zeta_{t,i})$,
\begin{equation}
    \E\brak*{\zeta_{t,i}^2 \one_{\set{\zeta_n > 0} \cap \set{D\hat U(x)_t \cdot \zeta_t > 0}}} = \frac{1}{2}\E\brak*{\zeta_{t,i}^2 \one_{\set{\zeta_n > 0}}} = \frac{1}{4}\sigma^2\,,
\end{equation}
and combining this with
\begin{gather}
    \abs*{\E\brak*{\zeta_{t,i}^2 \one_{\set{\zeta_n > -Ch} \cap \set{D\hat U(x)_t \cdot \zeta_t > 0}}} - \E\brak*{\zeta_{t,i}^2 \one_{\set{\zeta_n > 0} \cap \set{D\hat U(x)_t \cdot \zeta_t > 0}}}} \leq Ch\,,\\
    \abs*{\E\brak*{\zeta_{t,i}^2 \one_{\set{\zeta_n > -Ch} \cap \set{D\hat U(x)_t \cdot \zeta_t > -Ch}}}- \E\brak*{\zeta_{t,i}^2 \one_{\set{\zeta_n > -Ch} \cap \set{D\hat U(x)_t \cdot \zeta_t > 0}}}} \leq Ch^{0.75}
\end{gather}
yields
\begin{align}
    \E\brak*{\zeta_{t,i}^2\one_{E\cap \set{D>0}}} &\leq \E\brak*{\zeta_{t,i}^2 \one_{\set{\zeta_n > -Ch} \cap \set{D\hat U(x)_t \cdot \zeta_t > -Ch}}}\leq \frac{1}{4}\sigma^2 + Ch^{0.75}\,.
\end{align}
This completes the proof for~\eqref{e:zeta_t-exit-positive-dev-small} and~\eqref{e:zeta_n-exit-positive-dev-small}.
\end{proof}

\subsection{Proofs for the properties of~$\hat P$}\label{s:Phat-properties-proof}
In this subsection, we prove the properties of the perturbation~$\hat P$ stated in Lemmas~\ref{l:Utilde-vanish}--\ref{l:DU-lb-ind-ae}, which were used in the previous subsections. Since Lemma~\ref{l:facts-on-boundaries} collects standard results, we defer its proof to Appendix~\ref{a:regularities-of-basin-projection}. We begin by proving Lemma~\ref{l:Utilde-vanish}.

\begin{proof}[Proof of Lemma~\ref{l:Utilde-vanish}]
Fix~$x\in \partial\Omega \cap B(0, \rho a\sqrt{\epsilon})$. Since~$\partial\Omega \in C^5$ from Lemma~\ref{l:facts-on-boundaries} and it has bounded curvature from the compactness, there exists small~$t_0(x)$ such that for all~$\abs{t}\leq t_0$, 
\begin{equation}
    \xi(\gamma(t)) = x \quad\text{and}\quad \abs{\gamma(t)} < \rho a\sqrt{\epsilon}\,,
\end{equation}
where~$\gamma(t) \defeq x + n(x) t$ (See e.g., Lemma~\ref{l:unique_projection}). Then, for~$\abs{t} \leq t_0$, 
\begin{equation}
    g(t) \defeq \hat P(\gamma (t)) = \hat P(x, tn(x)) = \frac{1}{2}x^T K x \chi\paren*{\frac{\abs{x}^2}{a^2\epsilon}}\chi\paren*{\frac{t^2}{\tilde a^2 a^2 \epsilon}} = Af\paren*{t^2}\,,
\end{equation}
where
\begin{equation}
    A\defeq \frac{1}{2}x^T K x \chi\paren*{\frac{\abs{x}^2}{a^2\epsilon}}\quad\text{and}\quad f(s) = \chi\paren*{\frac{s}{\tilde a^2 a^2 \epsilon}}\,.
\end{equation}
We notice
\begin{align}
    D\hat P(x) \cdot n(x) &= g'(0) = 2tAf'\paren*{t^2}\vert_{t=0} = 0\,,\\
    n(x)^T D\hat P(x) n(x) &= g''(0) = 2Af'\paren*{t^2} + 4At^2 f''\paren*{t^2}\vert_{t=0} \\
    &= 2Af'(0) = \frac{2A}{\tilde a^2 a^2 \epsilon}\chi'(0) \overset{\eqref{e:chi-def}}{=} 0\,.
\end{align}
For~$x \in \partial\Omega \cap B(0,\rho a\sqrt{\epsilon})^c$, $\hat P = 0$ so~\eqref{e:DPs-normal-vanish} holds trivially. This completes the proof.
\end{proof}

A scaling argument yields the bounds for the $C^3$-norm of~$\hat P$ as stated in Lemma~\ref{l:P-regularity}.

\begin{proof}[Proof of Lemma~\ref{l:P-regularity}]
    Let~$\hat P \colon \R^d\times \R^d \to [0,\infty)$ be defined by
    \begin{equation}
         \hat P (\hat y, \hat z) = \frac12\, \hat y^T K \hat y \,\chi\paren*{\abs{\hat y}^2}\,\chi\paren*{\tilde a^{-2}\abs{\hat z}^2}\,,
    \end{equation}
    where~$K$, $\chi$, and~$\tilde a$ are defined as in Definition~\ref{d:P-perturb-def} and~\eqref{e:rho-def}.  

    Set
    \begin{equation}
        r=a\sqrt{\epsilon}\,,\quad \hat y = \frac{y}{r}\,,\quad \hat z = \frac{z}{r}\,,
    \end{equation}
    and recall that~$P$ is defined as in~\ref{e:P-def} and thus,
    \begin{equation}
        P(y,z) = r^2 \hat P (\hat y, \hat z)\,, \quad \forall y,z\in\R^d\,.
    \end{equation}

    Since $\norm{\hat P}_{C^3(\R^d)}$ is independent of~$a,\epsilon$, and
    \begin{equation}
        D_y \hat y = D_z \hat z = \frac{1}{r}I_d\,,
    \end{equation}
    it follows that for each~$0\leq i\leq 3$,
    \begin{equation}
        \norm{D^iP}_{L^\infty(\R^d)} 
        = r^{2-i}\norm{D^i \hat P}_{L^\infty(\R^d)} 
        \leq C r^{2-i}\,,
    \end{equation}
    for some constant~$C>0$ independent of~$a,\epsilon$.

    Combining this with the fact that~$\norm{\xi}_{C^4(\Omega_{r_0})}$ is independent of~$a,\epsilon$, and taking~$a\sqrt{\epsilon}$ sufficiently small if necessary, we obtain~\eqref{e:P-regularity}. 
    The bounds for~$\hat U$ then follow immediately from the definition~$\hat U = U - \hat P$.
\end{proof}

Before proving Lemmas~\ref{l:Utilde-saddle-almost-local-maximum} and~\ref{l:DU-lb-ind-ae}, we show that~$D\xi(0)$ is in fact the projection onto the stable subspace of~$H(0)=D^2U(0)$, using the identity~\eqref{e:Dd-normal-identity}. Consequently, $y=\xi(x)$ and~$z=x-\xi(x)$ lie approximately in the stable and unstable subspaces of~$H(0)$, respectively, up to an error of order~$O(\abs{x}^2)$.

\begin{lemma}\label{l:Dpi0-Ps}
    Let~$\xi\colon B(0, r_0) \to \partial\Omega$ be defined as in Lemma~\ref{l:facts-on-boundaries}. Define~$P_s$ as in~\eqref{e:Ps-def}. Then,
        \begin{equation}\label{e:Dpi0}
            D\xi(0) = P_s\,.
        \end{equation}
    Moreover, for any~$x \in B(0, r_0/2)$, if we denote
    \begin{equation}
        y = \xi(x) \quad\text{and}\quad z = x-\xi(x)\,,
    \end{equation}
    then
    \begin{gather}
        \label{e:y-almost-stable}
        y = P_s x + O(\abs{x}^2) = P_s y + O(\abs{x}^2)\,,\\
        \label{e:z-almost-unstable}
        z = P_u x + O(\abs{x}^2) = P_u z + O(\abs{x}^2)\,,\\
        \label{e:pi-almost-stable-projection}
        D\xi(x) = P_s + O(\abs{x})\,,
    \end{gather}
    where
    \begin{equation}\label{e:Pu-def}
        P_u = n(0)n(0)^T\,,
    \end{equation}
    is the projection onto the unstable eigenspace of~$H(0)$.

\end{lemma}
\begin{proof}[Proof of Lemma~\ref{l:Dpi0-Ps}]
Define the signed distance function~$d$ as in~\eqref{e:signed-d-def}. Recall that~$\xi \in C^4(\Gamma_{r_0}, \partial\Omega)$ by Lemma~\ref{l:facts-on-boundaries}. Differentiating both sides of the second identity in~\eqref{e:Dd-normal-identity} with respect to~$x$, using the first identity in~\eqref{e:Dd-normal-identity}, evaluating at~$\xi(x)$, and using~$d(\xi(x)) = 0$, we obtain
\begin{equation}\label{e:Dpipi-identity}
    D\xi(\xi(x)) = I_d - n(\xi(x))n(\xi(x))^T\,, \quad \forall x\in \Gamma_{r_0}\,.        
\end{equation}
In particular, since~$0\in \partial\Omega$, evaluating~\eqref{e:Dpipi-identity} at~$x=0$ yields~\eqref{e:Dpi0}.

Using the Taylor expansion of~$\xi$ at~$x=0$, and combining this with~\eqref{e:Dpi0}, we obtain
\begin{equation}\label{e:y-stable-of-x}
    y = P_s x + O(\abs{x}^2)\,, \quad \forall x\in B\paren*{0, \frac{r_0}{2}}\,,
\end{equation}
and, by the definition~$z = x - y$, we also obtain
\begin{equation}
     z = P_u x + O(\abs{x}^2)\,, \quad \forall x\in B\paren*{0, \frac{r_0}{2}}\,.
\end{equation}

Since~$P_s$ is a projection, we have~$\norm{P_s} \leq 1$ and~$P_s^2 = P_s$. Applying~$P_s$ to both sides of~\eqref{e:y-stable-of-x} yields
\begin{equation}
    P_s x = P_s y + O(\abs{x}^2)\,,
\end{equation}
and combining this with~\eqref{e:y-stable-of-x}, we obtain~\eqref{e:y-almost-stable}. 
A similar argument yields~\eqref{e:z-almost-unstable}.
\end{proof}

Lemma~\ref{l:Utilde-saddle-almost-local-maximum} is a direct consequence of the definition of~$\hat P$ together with property~\eqref{e:Dpi0}.
\begin{proof}[Proof of Lemma~\ref{l:Utilde-saddle-almost-local-maximum}]
Note that the regularity of~$\xi$~\eqref{e:pi-regularity} on the compact~$\partial\Omega$ yields that for some~$a, \epsilon$-independent constant~$C>0$ and any~$x\in B(0,r_0)$, it holds that $\abs{\xi(x)} \leq C\abs{x}$.
This implies that there exists~$c\in (0, \rho)$, independent of~$a,\epsilon$, such that for any~$x \in B(0, ca\sqrt{\epsilon})$,
\begin{equation}
    \abs{y} = \abs{\xi(x)} \leq \frac{1}{4}a\sqrt{\epsilon}\quad\text{and}\quad \abs{z} = \abs{x-\xi(x)} \leq \frac{1}{4}\tilde a a\sqrt{\epsilon}\,,
\end{equation}
and hence, noting that~$K$ is symmetric, 
\begin{multline}
    \hat P (x) = \frac{1}{2}y^T K y\,,\quad D\hat P(x)^T = y^T K D\xi(x) \,,\quad\text{and}\quad \\
    D^2\hat P(x) = D\xi(x)^T K D\xi(x) + \sum_{i=1}^d \paren*{K\xi(x)}_i D^2\xi_i(x)\,.
\end{multline}
Combining this with~$y(0)=\xi(0)=0$, \eqref{e:Dpi0}, and the fact that
\begin{equation}\label{e:KP-commutativity}
P_s^T=P_s\quad\text{and}\quad P_sK=KP_s=K\,,    
\end{equation}
 we obtain~\eqref{e:DPtilde-critical-minima}. Using~$$\hat U = U - P\,,\quad DU(0)=0\,,\quad D^2U(0) = H(0) = H_u + H_s\,,$$ and the definition of~$K$ in~\eqref{e:K-def} implies~\eqref{e:Utilde-almost-maximum-at-saddle}.
\end{proof}

Note that a Taylor expansion implies that for any~$C^3$ Morse function~$V$ satisfying~$V(0)=0$, $DV(0)=0$, and such that~$D^2V(0)$ has eigenvalues~$-\bar \lambda_u < 0 < \bar \lambda_1 \le \ldots \le \bar \lambda_{d-1}$, there exists a sufficiently small~$\bar r>0$ such that for all~$x\in B(0, \bar r)$,
\begin{equation}\label{e:morse-function-DV-lb}
\abs{DV(x)} \geq C_V \abs{x},,
\quad\text{where}\quad
C_V = \frac{1}{2}\min\set*{\bar\lambda_u, \bar\lambda_1},.
\end{equation}

However, a naive application of Taylor expansion to~$\hat U$ only guarantees the property~\eqref{e:DU-lb} within a ball of radius~$O(a\sqrt{\epsilon})$, since the $C^3$-norm of~$\hat U$ is of order~$O(1/(a\sqrt{\epsilon}))$, as shown in~\eqref{e:U-regularity}. This makes it difficult to ensure that the gradient of~$\hat U$ remains sufficiently large before~$\Delta \hat U$ increases from its negative value at the saddle~$\Delta \hat U(0)$.

The key feature of the construction of~$\hat P$ is that it separates the normal and tangential components of~$x$, and hence the stable and unstable subspaces of~$H(0)$, up to a small error, as described in Lemma~\ref{l:Dpi0-Ps}. This allows us to analyze the norm of~$D\hat U$ on these orthogonal subspaces separately. As a result, we recover a property analogous to~\eqref{e:morse-function-DV-lb} on an $O(1)$-radius neighborhood, with constant~$C_V = \frac{1}{2}\min\set{\kappa, \lambda_u}$, where~$\kappa$ and~$-\lambda_u$ are the eigenvalues of~$D^2\hat U(0)$ as given in~\eqref{e:Utilde-almost-maximum-at-saddle}..

\begin{proof}[Proof of Lemma~\ref{l:DU-lb-ind-ae}]
Throughout the proof, we adopt the following notational conventions. 
For a scalar function~$f\in C^\infty(\R^d, \R)$, we write $Df = \brak{\partial_1 f\ldots \partial_d f}$ and view it as an element of~$\R^{1\times d}$. 
Moreover, we write~$R_1=O(R_2)$ to mean that there exists an $a,\epsilon$-independent constant~$C>0$ such that~$\abs{R_1} \leq C\abs{R_2}$.
    
Let~$P_s$ and~$H_s$ be defined as in~\eqref{e:Ps-def} and~\eqref{e:Hs-def}, respectively. Similarly, define~$P_u$ and~$H_u$ as in~\eqref{e:Pu-def} and~\eqref{e:Hu-def}. 
We note that~$P_s^T = P_s$, $P_s^2 = P_s$, and $H_s = H_sP_s = P_sH_s$, and analogous properties hold for~$P_u$.

For any~$x\in B(0,1)$,
\begin{equation}\label{e:DtildeU-pythagoras}
\abs{D\hat U(x)}^2 = \abs{P_s D\hat U(x)^T}^2 + \abs{P_u D\hat U(x)^T}^2\,,
\end{equation}
and we estimate each term separately.

For any~$x\in B(0,r_0/2)$, using~$DU(0)=0$ and a Taylor expansion of~$DU(x)$, we obtain
\begin{equation}\label{e:DtildeU-taylor-perturb}
D\hat  U(x)^T = DU(x)^T - D\hat P(x)^T = H(0)x -D\hat P(x)^T + O(\abs{x}^2)\,,
\end{equation}
which implies
\begin{align}
P_s D\hat U(x)^T &= H_sP_s x - P_s D\hat P(x)^T + O(\abs{x}^2)\\
\label{e:DtildeU-stable-part}
&\overset{\eqref{e:y-almost-stable}}{=} H_sP_sy - P_sD\hat P(x)^T + O(\abs{x}^2)\,.
\end{align}

For any~$x\in B(0, \rho a\sqrt{\epsilon})$, we have
\begin{equation}\label{e:DtildeP-chain-rule}
D\hat P (x)^T = (D_x y)^T(D_yP)^T + (D_x z)^T (D_z P)^T\,,
\end{equation}
with
\begin{equation}
(D_y P)^T = M_y y\,, \quad (D_z P)^T = M_z z\,,
\end{equation}
where~$M_y, M_z$ are symmetric matrices in~$\R^{d\times d}$ given by
\begin{align}
M_y &= \chi(t)\chi(s)K + \frac{y^T K y}{a^2\epsilon} \chi'(t)\chi(s) I_d\,,\\
M_z &= \frac{y^T K y}{\tilde a^2 a^2 \epsilon}\chi(t)\chi'(s)I_d\,,
\end{align}
and
\begin{equation}
t = \frac{\abs{y}^2}{a^2\epsilon}\,, \quad s = \frac{\abs{z}^2}{\tilde a^2 a^2 \epsilon}\,.
\end{equation}

Combining this with~\eqref{e:DtildeP-chain-rule}, and using
\begin{align}
(D_x y)^T &= (D\xi(x))^T \overset{\eqref{e:pi-almost-stable-projection}}{=} P_s + O(\abs{x})\,,\\
(D_x z)^T &= I_d - (D_x y)^T = P_u + O(\abs{x})\,,
\end{align}
together with $M_y, M_z = O(1)$ and $y, z = O(\abs{x})$, we obtain
\begin{equation}\label{e:DtildeP-stable-unstable}
D\hat P(x)^T = P_s M_y y + P_u M_z z + O(\abs{x}^2)\,.
\end{equation}

Multiplying both sides by~$P_s$ and using $P_sP_u =0$, \eqref{e:KP-commutativity}, and~\eqref{e:K-def}, we obtain
\begin{align}
P_sD\hat P(x)^T 
&= M_yP_s y + O(\abs{x}^2)\\
&= \paren*{\chi(t)\chi(s) H_s - \kappa \chi(t)\chi(s)P_s + \frac{y^TKy}{a^2\epsilon}\chi'(t)\chi(s) P_s} P_s y + O(\abs{x}^2)\,. 
\end{align}

Substituting into~\eqref{e:DtildeU-stable-part}, we obtain
\begin{equation}\label{e:PsDtildeU-My}
P_sD\hat U(x)^T = \tilde M_yP_sy + O(\abs{x}^2)\,,
\end{equation}
where
\begin{equation}
\tilde M_y = \paren*{1-\chi(t)\chi(s)}H_s + \kappa \chi(t)\chi(s)P_s - \frac{y^TKy}{a^2\epsilon}\chi'(t)\chi(s)P_s\,.
\end{equation}

Since $0\leq \chi(t)\chi(s)\leq 1$, the first two terms form a convex combination of two symmetric matrices $H_s$ and $\kappa P_s$. 
Moreover, since $\chi \geq 0$, $\chi' \leq 0$, and $y^T K y \geq 0$, the third term is positive semidefinite. 
Therefore,
\begin{equation}\label{e:MtildeyPsy}
\abs{\tilde M_y P_s y} \geq \kappa \abs{P_s y}\,.
\end{equation}

Similarly,
\begin{equation}
P_u D\hat U(x)^T = H_uP_u z - P_uD\hat P(x)^T + O(\abs{x}^2)\,,
\end{equation}
and combining with~\eqref{e:DtildeP-stable-unstable} yields
\begin{equation}\label{e:PuDtildeU-Mz}
P_u D\hat U(x)^T = \tilde M_z P_u z + O(\abs{x}^2)\,,
\end{equation}
where
\begin{equation}\label{e:MtildezPuz}
\tilde M_z = H_u - \frac{y^T K y}{\tilde a^2 a^2 \epsilon}\chi(t)\chi'(s)P_u\,.
\end{equation}

Using the definition of~$\tilde a$ in~\eqref{e:rho-def}, $K \preceq \lambda_{d-1}I_d$, and $\supp\chi \subset [0,1]$, we obtain
\begin{equation}
-\lambda_u - \frac{y^T K y}{\tilde a^2 a^2 \epsilon}\chi(t)\chi'(s) 
\leq -\lambda_u + \frac{\lambda_{d-1}}{\tilde a^2}\norm{\chi'}_{L^\infty} 
\leq -\frac{1}{2}\lambda_u\,,
\end{equation}
and hence
\begin{equation}
\abs{\tilde M_z P_u z}\geq \frac{1}{2}\lambda_u\abs{P_u z}\,.
\end{equation}

Combining~\eqref{e:DtildeU-pythagoras}, \eqref{e:PsDtildeU-My}, \eqref{e:MtildeyPsy}, \eqref{e:PuDtildeU-Mz}, and~\eqref{e:MtildezPuz}, we obtain
\begin{align}
\abs{D\hat U(x)} 
&\geq \frac{1}{\sqrt{2}} \paren*{\abs*{P_s D\hat U(x)^T} + \abs*{P_uD\hat U(x)^T} } \\
&\geq \frac{1}{\sqrt{2}}\paren*{\kappa \abs{P_s y} + \frac{1}{2}\lambda_u \abs{P_u z}} - O(\abs{x}^2)\\
&\overset{\eqref{e:y-almost-stable}, \eqref{e:z-almost-unstable}}{\geq} c(\kappa, \lambda_u) \paren*{\abs{y}+\abs{z}} - O(\abs{x}^2)\\
\label{e:DtildeU-lb-small-radius}
&\overset{x=y+z}{\geq} c_0\abs{x} - c_1\abs{x}^2\,.
\end{align}

This holds for all $x\in B(0, \rho a\sqrt{\epsilon})$. Reducing~$a\sqrt{\epsilon}$ if necessary yields
\begin{equation}
c_0\abs{x} - c_1\abs{x}^2 
\geq \frac{1}{2}c_0 \abs{x}\,, \quad \forall x\in B(0, \rho a\sqrt{\epsilon})\,,
\end{equation}
and hence
\begin{equation}\label{e:DtildeU-lb-small-rho}
\abs{D\hat U(x)} \geq \frac{1}{2}c_0\abs{x}\,, \quad \forall x\in B(0, \rho a\sqrt{\epsilon})\,.
\end{equation}

Outside $B(0, \rho a\sqrt{\epsilon})$, we have $\hat P = 0$, so $\hat U = U$. 
Moreover, for some small~$r_2\in (0,1)$,
\begin{align}\label{e:DU-lb-small-r2}
\abs{DU(x)} 
&= \abs{H(0)x} - \norm{D^3U}_{L^\infty(B(0,1))}\abs{x}^2\\
&\geq \min\set{\lambda_1, \lambda_u} \abs{x} - C\abs{x}^2
\geq c\abs{x}\,, \quad \forall x\in B(0,r_2)\,.
\end{align}

Finally, by choosing $a\sqrt{\epsilon}$ sufficiently small so that $\rho a\sqrt{\epsilon} < r_2$, we can combine~\eqref{e:DtildeU-lb-small-rho} and~\eqref{e:DU-lb-small-r2} to conclude~\eqref{e:DU-lb}.
\end{proof}

\section{Spectral gap of local chain (Proof of Lemmas~\ref{l:sg-restricted} and~\ref{l:large-temperature-gap})}\label{s:sg-restricted}
In this section, we prove the two key lemmas stated in Section~\ref{s:key-lemmas}, namely Lemmas~\ref{l:sg-restricted} and~\ref{l:large-temperature-gap}. 

\subsection{Proof of Lemma~\ref{l:sg-restricted}}\label{ss:sg-restricted-proof}
In this subsection, we prove Lemma~\ref{l:sg-restricted}, which provides a lower bound for the spectral gap of the restricted chain in the small-temperature regime. We first outline the main idea, then state the auxiliary lemmas, and finally complete the proof of Lemma~\ref{l:sg-restricted}, postponing the proofs of the intermediate lemmas to the end of this section.

To prove Lemma~\ref{l:sg-restricted}, we first use the Lyapunov drift condition~\eqref{e:lyapunov} to relate the spectral gap of the restricted chain~$\hat Q_{h, \epsilon} = \hat P_{h, \epsilon}\vert_{\Omega_1}$ to that of~$\hat P_{h, \epsilon}\vert_{B(m_1, a \sqrt{\epsilon})}$, i.e., the chain further restricted to a neighborhood of the local minimum~$m_1$ (Lemma~\ref{l:restrict-to-minimum}). 

Next, by applying the Holley--Stroock Lemma~\ref{l:holley-stroock}, we compare the spectral gaps of~$\hat P_{h, \epsilon}\vert_{B(m_1, a \sqrt{\epsilon})}$ and~$\bar P_{h, \epsilon}\vert_{B(m_1, a \sqrt{\epsilon})}$, where the latter denotes the Metropolis random walk with step size~$h$ and Gaussian stationary distribution, restricted to~$B(m_1, a\sqrt{\epsilon})$ (Lemma~\ref{l:minimum-equiv-normal}). 

Since the spectral gap of a Metropolis random walk with a Gaussian stationary distribution on a convex set is well understood (see, e.g.,~\cite{LovaszSimonovits93, KannanLi96, CousinsVempala14}), we obtain a lower bound for the spectral gap of~$\hat Q_{h, \epsilon}$ (Lemma~\ref{l:mrw-normal-gap}). Finally, using the smallness of the perturbation~\eqref{e:L^infty-perturb} together with another application of the Holley--Stroock Lemma~\ref{l:holley-stroock}, we transfer this bound to~$Q_{h, \epsilon}$.

We now formally state the lemmas described above and prove Lemma~\ref{l:sg-restricted}.
\begin{lemma}\label{l:restrict-to-minimum}
Let~$\hat\epsilon, \eta, \lambda, \gamma, a, b$ be as in Lemma~\ref{l:lyapunov}.  Then, for all~$\epsilon \le \hat\epsilon$ and all~$h$ satisfying~$0<h/\epsilon^2 \le \eta$, we have
    \begin{equation}\label{e:restrict-to-minimum}
        \gap(\hat P_{h, \epsilon}\vert_{\Omega_1}) \geq  \frac{\alpha \lambda \gamma h^2}{b + \alpha}\,,\quad\text{where}\quad \alpha = \gap(\hat P_{h, \epsilon}\vert_{B(m_1, a\sqrt{\epsilon})})\,.
    \end{equation}
\end{lemma}

Using Holley-Stroock~\cite[Proposition 2.3]{MadrasPiccioni99} and~\cite[Theorem 3.1]{KannanLi96}, we have the following.
\begin{lemma}\label{l:minimum-equiv-normal}
Let~$\hat\epsilon, \eta, a$ be as in Lemma~\ref{l:lyapunov}.  Then, for all~$\epsilon \le \hat\epsilon$ and all~$h$ satisfying~$0<h/\epsilon^2 \le \eta$, we have
    \begin{equation}\label{e:minimum-equiv-normal}
        \gap(\hat P_{h, \epsilon}\vert_{B(m_1, a\sqrt{\epsilon})}) \geq 2^{-4} \gap(\bar P_{h, \epsilon}\vert_{B(m_1, a\sqrt{\epsilon})})\,.
    \end{equation}
    Here, $\bar P_{h, \epsilon}\vert_{B(m_1, a\sqrt{\epsilon})}$ denotes the restriction of the chain $\bar P_{h,\epsilon}$ to $B(m_1, a\sqrt{\epsilon})$, where $\bar P_{h,\epsilon}$ is Metropolis random walk with step size~$h$ and the normal stationary distribution~$N(m_1, \epsilon D^2 U(m_1)^{-1})$.
\end{lemma}

\begin{lemma}\label{l:mrw-normal-gap}
Let~$\hat\epsilon, \eta, a$ be as in Lemma~\ref{l:lyapunov}.  Then, there exists a constant~$c>0$ such that for all~$\epsilon \le \hat\epsilon$ and~$0 < h / \epsilon^2\leq \eta $,
    \begin{equation}\label{e:mrw-normal-gap}
        \gap(\bar P_{h, \epsilon}\vert_{B(m_1, a\sqrt{\epsilon})}) \geq c\frac{h^2}{\epsilon}\,.
    \end{equation}
Here, $\bar P_{h, \epsilon}\vert_{B(m_1, a\sqrt{\epsilon})}$ is defined as in Lemma~\ref{l:minimum-equiv-normal}.
\end{lemma}

 Finally, we can prove Lemma~\ref{l:sg-restricted}.
\begin{proof}[Proof of Lemma~\ref{l:sg-restricted}]
    Define~$\hat Q_{h, \epsilon} = \hat P_{h, \epsilon}\vert_{\Omega_1}$ as in Lemma~\ref{l:lyapunov}. 
    Combining~\eqref{e:restrict-to-minimum}, \eqref{e:minimum-equiv-normal}, and~\eqref{e:mrw-normal-gap} in  Lemma~\ref{l:restrict-to-minimum}--\ref{l:mrw-normal-gap}, we have
    \begin{equation}\label{e:qtilde-he-lb}
          \gap(\hat Q_{h, \epsilon}) \geq \frac{\alpha\lambda \gamma h^2}{b+\alpha} \geq \frac{\lambda \gamma}{b+2} \alpha h^2\,,
    \end{equation}
    where~$\alpha = \gap(\hat P_{h, \epsilon}\vert_{B(m_1, a\sqrt{\epsilon})})$ satisfies
    \begin{equation}\label{e:alpha-lb}
        \alpha \geq c\frac{h^2}{\epsilon}\,.
    \end{equation}
    Moreover, using Holley-Stroock lemma~\ref{l:holley-stroock} and the bound~\eqref{e:L^infty-perturb} yields
    \begin{equation}\label{e:qhe-qtildehe-equiv}
        \gap(Q_{h, \epsilon}) \geq c \gap(\hat Q_{h, \epsilon})\,.
    \end{equation}
    Combining~\eqref{e:qtilde-he-lb}, \eqref{e:alpha-lb}, and~\eqref{e:qhe-qtildehe-equiv}, we obtain~\eqref{e:sg-restricted}.
\end{proof}

\subsection{Proof of Lemma~\ref{l:large-temperature-gap}}\label{s:large-temperature-gap-proof}
In this subsection, we prove Lemma~\ref{l:large-temperature-gap}. The proof relies on the Holley--Stroock Lemma~\ref{l:holley-stroock} and the definition of the spectral gap for a reversible Markov chain given in~\eqref{e:gap-def}.

\begin{proof}[Proof of Lemma~\ref{l:large-temperature-gap}]
    Define~$ P_{\eta\hat\epsilon^2, \epsilon}\vert_{\Omega}$ as the Metropolis random walk with step size~$\eta\hat\epsilon^2$ and the stationary distribution~$\pi_{\epsilon}\vert_{\Omega}$. Then, we observe that 
    \begin{equation}
       1\leq \frac{\tilde\pi_\epsilon}{\tilde \pi_{\hat\epsilon}} = \exp\paren*{\paren*{\frac{1}{\hat\epsilon} - \frac{1}{\epsilon}} U} \leq  C(\hat\epsilon)\defeq \exp\paren*{\frac{1}{\hat\epsilon} \norm{U}_\infty}\,,
    \end{equation}
    and hence by Holley-Stroock lemma~\ref{l:holley-stroock}, 
    \begin{equation}\label{e:ek-e0-same-stepsize}
        \gap(P_{\eta\hat\epsilon^2, \epsilon}\vert_{\Omega}) \geq C(\hat\epsilon)^{-2} \gap(P_{\eta \hat\epsilon^2, \hat\epsilon}\vert_{\Omega})\,.
    \end{equation}
    Then, we recall
    \begin{equation}\label{e:gap-pk-def}
        \gap(Q_{h, \epsilon})=\gap(P_{h, \epsilon}\vert_{\Omega}) = \inf_{f\in L^2(\pi_\epsilon)\setminus\set{0}} \frac{1}{2\var_{\pi_\epsilon}(f)}D(f)\,,
    \end{equation}
    where
    \begin{equation}
        D(f) = \frac{1}{\pi_\epsilon(\Omega)}\int_{\Omega}\int_{\Omega} \abs{f(x)-f(y)}^2  Q_{h, \epsilon}(x,dy) \pi_\epsilon(x)dydx\,.
    \end{equation}
    We see
    \begin{align}
        D(f) &= \frac{1}{\pi_\epsilon(\Omega)} \int_{\Omega}\int_{\Omega\setminus\set{x}} \abs{f(x)-f(y)}^2   Q_{h, \epsilon}(x,dy) \pi_\epsilon(x)dydx\\
        &= \frac{1}{\pi_\epsilon(\Omega)} \int_{\Omega} \frac{1}{\abs{B(x, h)}} \int_{B(x, h)\cap \paren*{\Omega\setminus\set{x}}} g_f(x,y) dydx\,,
    \end{align}
    where $g_f(x,y) = \abs{f(x)-f(y)}^2  \min\set*{\pi_\epsilon(x), \pi_\epsilon(y)}$,
    and using~$h \geq \eta\hat\epsilon^2$ from the assumption, we obtain
    \begin{align}
        D(f) &\geq  \frac{1}{\pi_\epsilon(\Omega)} \paren*{\frac{\eta\hat\epsilon^2}{h}}^{d}\int_{\Omega} \frac{1}{\abs{B(x,  \eta \hat\epsilon^2 )}} \int_{B(x, \eta\hat\epsilon^2)\cap \paren*{\Omega\setminus\set{x}}} g_f(x,y) dydx\\
        \label{e:Pk-Pktilde-gap-comparison}
        &= \paren*{\frac{\eta\hat\epsilon^2}{h}}^{d} \hat D(f) \geq \paren*{\frac{\eta\hat\epsilon^2}{\bar h}}^{d} \hat D(f)\,,
    \end{align}
    where 
    \begin{equation}
        \hat D(f) = \frac{1}{\pi_\epsilon(\Omega)} \int_{\Omega}\int_{\Omega} \abs{f(x)-f(y)}^2  P_{\eta\hat\epsilon^2, \epsilon}\vert_{\Omega}(x,dy) \pi_\epsilon(x)dydx\,.
    \end{equation}
    Combining~\eqref{e:gap-pk-def}, \eqref{e:Pk-Pktilde-gap-comparison}, \eqref{e:ek-e0-same-stepsize}, and using Lemma~\ref{l:sg-restricted} at~$\epsilon = \hat\epsilon$ and~$h = \eta \hat\epsilon^2$, we obtain
    \begin{align}
        \gap(Q_{h, \epsilon}) \geq \paren*{\frac{\eta\hat\epsilon^2}{\bar h}}^{d} \gap (P_{\eta\hat\epsilon^2, \epsilon}\vert_{\Omega}) &\geq \paren*{\frac{\eta\hat\epsilon^2}{\bar h}}^{d}  C(\hat\epsilon)^{-2}\gap (P_{\eta\hat\epsilon^2, \hat\epsilon}\vert_{\Omega})\\
        &\geq  \hat c_1 \eta^4 \paren*{\frac{\eta\hat\epsilon^2}{\bar h}}^{d}  C(\hat\epsilon)^{-2} \hat\epsilon^3 \defeq  c_2(\eta, \hat\epsilon, \bar h)\,,
    \end{align}
    and this completes the proof for~\eqref{e:large-temperature-gap}.
\end{proof}

\subsection{Proofs of Lemma~\ref{l:restrict-to-minimum}--\ref{l:mrw-normal-gap}}\label{s:restricted-to-minimum--normal-gap-proof}
It remains to prove Lemmas~\ref{l:restrict-to-minimum}--\ref{l:mrw-normal-gap}. To establish the former, we use the following lemma, which connects the spectral gap of a Metropolis chain to that of its restriction under a Lyapunov drift condition. For continuity, we postpone its proof to Appendix~\ref{a:tools-for-gap}.

\begin{lemma}[Lyapunov condition for the spectral gap of a Metropolis chain]\label{l:mrw-lyapunov}
    Let~$(\mathcal X, \mathcal B, m)$ be a Polish measure space, and let~$\pi$ and~$q(x,\cdot)$ be probability densities on~$\mathcal X$ for each~$x\in \mathcal X$. Let~$P(x,dy)$ be the transition kernel of a Metropolis chain with proposal~$q$ and stationary measure~$\pi$. Suppose that there exist constants $\lambda_1, b_1 > 0$, a measurable set $K \subset \mathcal X$, and a measurable function $V \colon \mathcal X \to [1,\infty)$ such that
\begin{equation}
PV \le (1-\lambda_1) V + b_1 \one_K \,.
\end{equation}
Let $P\vert_K$ denote the restriction of $P$ to $K$, as defined in Definition~\ref{d:Markov-chain-restriction}. Then
\begin{equation}\label{e:gap-lb-lyapunov}
\gap(P)
\ge
\frac{\alpha_1 \lambda_1}{b_1+\alpha_1},
\qquad
\text{where}
\qquad
\alpha_1 = \gap(P\vert_K).
\end{equation}
\end{lemma}

\begin{proof}[Proof of Lemma~\ref{l:restrict-to-minimum}]
   This is direct from applying Lemma~\ref{l:mrw-lyapunov} to~$P = \hat Q_{h, \epsilon}$, which is the Metropolis random walk with step size~$h$ and the stationary density~$\hat \pi_\epsilon \vert_{\Omega_1}$, with the Lyapunov drift condition~\eqref{e:lyapunov}. 
\end{proof}

To prove Lemma~\ref{l:minimum-equiv-normal}, we use the fact that the Morse potential~$U$ is quadratic with a positive-definite Hessian in a neighborhood of the local minimum. This allows us to compare the Metropolis random walk with Gibbs stationary distribution to that with an appropriately scaled normal stationary distribution, via the Holley--Stroock lemma.

\begin{proof}[Proof of Lemma~\ref{l:minimum-equiv-normal}]
Using~\eqref{e:Utilde-equals-U} and the Taylor expansion of~$U$, we obtain that for all~$x\in B(m_1, a \sqrt{\epsilon})$,
\begin{equation}
    \hat U(x) = U(x) = (x-m_1)^T D^2U(m_1) (x-m_1) + O\paren*{ \epsilon^{3/2} }\,.
\end{equation}
This implies that, if we denote by~$\tilde \pi_\epsilon$ and~$\bar p_\epsilon$ the unnormalized densities of the Gibbs measure (defined as in~\eqref{e:gibbs-d-def}) and the normal distribution~$N(m_1, \epsilon D^2 U(m_1)^{-1})$, respectively, then by decreasing~$\hat\epsilon$ if necessary, for all~$\epsilon \leq \hat\epsilon$,
\begin{equation}
    \frac{1}{2} \leq \exp\paren*{-C\epsilon^{1/2}} \leq \frac{\tilde \pi_\epsilon}{\bar p_\epsilon} \leq \exp\paren*{C\epsilon^{1/2}} \leq 2\,.
\end{equation}
Hence, if we use the same step size~$h$ for both random walks, the Holley--Stroock Lemma~\ref{l:holley-stroock} implies~\eqref{e:minimum-equiv-normal}.
\end{proof}

As mentioned at the beginning of Subsection~\ref{ss:sg-restricted-proof}, the spectral gap of the Metropolis random walk on a convex set with a log-concave stationary density has been well studied (see, e.g.,~\cite{LovaszSimonovits93, KannanLi96, CousinsVempala14}). We apply these existing results to prove Lemma~\ref{l:mrw-normal-gap}.

\begin{proof}[Proof of Lemma~\ref{l:mrw-normal-gap}]
    Let~$\bar p_\epsilon$ be the unnormalized density of the normal distribution~$N(m_1, \epsilon D^2 U(m_1)^{-1})$, and for any~$z\in B(m_1, a \sqrt{\epsilon})$, 
    We observe that for any~$y \in B(m_1, a\sqrt{\epsilon})$,
    \begin{equation}\label{e:hatp-ub-lb}
        \exp\paren*{-a^2\theta_d} \leq \bar p_\epsilon (y) \leq \exp\paren*{-a^2\theta_1}\,,
    \end{equation}
    where $0<\theta_1<\theta_d$ are the smallest and largest eigenvalues of~$D^2 U(m_1)$, respectively. 

    Applying~\cite[Theorem 3.1]{KannanLi96} (or~\cite[Corollary 3.3]{LovaszSimonovits93}, \cite[Theorem 4.5.1]{Woodard07}) yields
    \begin{equation}\label{e:normal-mrw-gap-lb}
        \gap(\bar P_{h, \epsilon}\vert_{B(m_1, a\sqrt{\epsilon})}) \geq \frac{l^4 h^2 \theta_1 }{8d\pi \epsilon}\,,
    \end{equation}
    where the local conductance~$l$ is defined by
    \begin{equation}
        l \defeq \inf_{z\in B(m_1, a \sqrt{\epsilon})} \bar p(z)\,, 
    \end{equation}
    and
    \begin{equation}
         \bar p(z) = \P\brak*{X_1(z) \neq z} 
         = \abs{B(z, h)}^{-1} \int_{B(z, h) \cap B(m_1, a \sqrt{\epsilon})} 
         \min\set*{1, \frac{\bar p_\epsilon (y)}{\bar p_\epsilon(z)}}  \, dy\,.
    \end{equation}
    Here, $X_1(z) \sim \bar P_{h, \epsilon}\vert_{B(m_1, a\sqrt{\epsilon})}(z,\cdot)$ denote the random state of the Markov chain at time~$1$ with transition kernel~$\bar P_{h, \epsilon}\vert_{B(m_1, a\sqrt{\epsilon})}$.

    Using~\eqref{e:hatp-ub-lb}, we obtain
    \begin{align}\label{e:local-conductance-normal-lb}
        l 
        &\geq \exp\paren*{-a^2\paren*{\theta_d-\theta_1}} 
        \inf_{z\in B(m_1, a \sqrt{\epsilon})} 
        \frac{\abs{B(z, h) \cap B(m_1, a\sqrt{\epsilon})}}{\abs{B(z, h)}}\\
        &\geq C(a, \theta_d, \theta_1) 
        \paren*{\frac{1}{2}- \frac{\sqrt{d}h}{4a\sqrt{\epsilon}}}
        \geq C \paren*{\frac{1}{2} - O\paren*{\eta\epsilon^{3/2}}}
        \geq C\,,
    \end{align}
    for some $\epsilon$-independent constant~$C$, provided that $\hat\epsilon$ and~$\eta$ are sufficiently small.
    The second inequality follows from a standard geometric lemma on intersections of balls (see, e.g., \cite[Lemma 0.1]{LovaszSimonovits93} or~\cite[Lemma 4.5.1]{Woodard07}).
    Combining~\eqref{e:local-conductance-normal-lb} with~\eqref{e:normal-mrw-gap-lb} yields~\eqref{e:mrw-normal-gap}.
\end{proof}

\section{Overlap and first-level mixing estimates (Proof of Theorem~\ref{t:main-theorem})}\label{s:overlap}

In this section, we prove Theorem~\ref{t:main-theorem}. To this end, we apply~\cite[Theorem 3.1]{WoodardSchmidlerEA09}, for which we need to introduce the quantities~$\gamma_\pt$ and~$\delta_\pt$ appearing in its estimates, associated with a given sequence of probability measures~$(\pi_k)_{k=0}^N$. These are defined by
\begin{align}
    \gamma_\pt &\defeq \min_{i \in \set{1,2}}\prod_{k=1}^N \min\set*{1, \frac{\pi_{k-1}(\Omega_i)}{\pi_k(\Omega_i)}}\,,\\
    \delta_\pt &\defeq \min_{\substack{|k-l|=1 \\ i \in \set{1,2}}} \frac{1}{\pi_k(\Omega_i)}\int_{\Omega_i} \min\set*{\pi_k(x), \pi_l(x)} \, dx\,,
\end{align}
where~$\pi_k(\Omega_i) = \int_{\Omega_i} \pi_k(x)\,dx$.

Throughout the remainder of this section, given a sequence of temperatures $(\epsilon_k)_{k=0}^N$, we write~$\pi_k$, $\tilde \pi_k$, and~$Z_k$ in place of~$\pi_{\epsilon_k}$, $\tilde \pi_{\epsilon_k}$, and~$Z_{\epsilon_k}$, respectively, for notational convenience.

The next three lemmas provide the key ingredients needed to apply~\cite[Theorem 3.1]{WoodardSchmidlerEA09}. The first lemma establishes a lower bound for~$\gamma_\pt$. The second lemma provides a lower bound for the overlap quantity~$\delta_\pt$. The third lemma gives a lower bound for the spectral gap of a general lazy Metropolis random walk.

\begin{lemma}\label{l:gamma-swc-lb}
Suppose the potential~$U$ satisfies Assumptions~\ref{a:criticalpts} and~\ref{a:massRatioBound}. If all the temperatures~$(\epsilon_k)_{k=0}^N$ are in~$[\underline \theta, \bar \theta]$, then there exists a finite constant
\begin{equation}\label{e:C_BV-def}
    \hat C_\mathrm{BV} \defeq \max_{i\in \set{1,2}}\int_{\underline\theta}^{\bar \theta} \abs*{\partial_{\epsilon'} \pi_{\epsilon'}\paren*{\Omega_i}}d\epsilon' < \infty\,,
\end{equation}
such that
    \begin{equation}\label{e:gamma-swc-lb}
        \gamma_\pt  \geq \exp\paren*{-C_m^2 \hat C_\mathrm{BV}}\,,
    \end{equation}
\end{lemma}

\begin{proof}[Proof of Lemma~\ref{l:gamma-swc-lb}]
    Denote~$r_k = \pi_{k-1}\paren*{\Omega_1}/\pi_k\paren*{\Omega_1}$.
    We observe that
    \begin{align}
        \prod_{k=1}^N \paren*{1\wedge r_k} = \exp\paren*{\sum_{k=1}^N \log\paren*{1\wedge r_k}} = \exp\paren*{-\sum_{k=1}^N \log\paren*{1\vee r_k^{-1}}}\,,
    \end{align}
    and using the inequality~$\log\paren*{1\vee x}\leq \abs{x-1}$ yields
    \begin{align}
        \sum_{k=1}^N \log\paren*{1\vee r_k^{-1}} \leq \sum_{k=1}^N \abs*{r_k^{-1}-1} = \sum_{k=1}^N \pi_{k-1}\paren*{\Omega_1}^{-1} \abs*{\pi_{k-1}\paren*{\Omega_1}-\pi_{k}\paren*{\Omega_1}}\,.
    \end{align}
    Combining this Assumption~\ref{a:massRatioBound} with~\cite[Lemma 8.2]{HanIyerEA26} (Or see~\cite[Remark2.5]{HanIyerEA26}) implies
    \begin{equation}
        \sum_{k=1}^N \log\paren*{1\vee r_k^{-1}} \leq C_m^2 \int_{\underline\theta}^{\bar \theta} \abs*{\partial_{\epsilon'} \pi_{\epsilon'}\paren*{\Omega_1}}d\epsilon' = C_m^2 \hat C_\mathrm{BV}\,.
    \end{equation}
    Same bound holds for the case of~$\Omega_2$ so this completes the proof.
\end{proof}

\begin{lemma}\label{l:delta-swc-lb}
    Suppose~$\underline \epsilon < \bar \epsilon$ and~$\nu \in (0, 1/\underline \epsilon)$. Set
    \begin{equation}\label{e:N-e0-eN-def-delta}
        N=\ceil*{1/\bar\nu\underline \epsilon}\,,\quad \epsilon_0 = \bar \epsilon\,,\quad \text{and}\quad \epsilon_N = \underline \epsilon\,,    
    \end{equation}
    and let
    $\paren*{1/\epsilon_k}_{k=0}^{N}$ be linearly spaced Then, 
    \begin{equation}\label{e:delta-swc-lb}
        \delta_\pt\geq M^{-1}\,, \quad\text{where}\quad M(U, \bar\nu) \defeq \exp\paren*{\bar\nu \norm{U}_{L^\infty}}\,.
    \end{equation}
\end{lemma}
\begin{proof}[Proof of Lemma~\ref{l:delta-swc-lb}]

We observe that if~$\paren*{1/\epsilon_k}_{k=0}^{N}$ are linearly spaced, then the choice of~$N$ in~\eqref{e:N-e0-eN-def-delta} implies 
\begin{equation}
    0\leq \frac{1}{\epsilon_{k+1}} - \frac{1}{\epsilon_k} = \frac{1}{N}\paren*{\frac{1}{\underline \epsilon} - \frac{1}{\bar \epsilon}} \leq \bar\nu\,. 
\end{equation} 
This implies that
\begin{gather}
      M^{-1} \leq \frac{\tilde \pi_{k+1}(z)}{ \tilde \pi_k(z)} = \exp\paren*{-\paren*{\frac{1}{\epsilon_{k+1}}  -\frac{1}{\epsilon_{k}} } U(z)} \leq 1 \quad\text{and}\quad M^{-1} \leq \frac{Z_{k+1}}{Z_k} \leq 1\,,
\end{gather}
and hence, for each~$k\in\set{0, \ldots, N}$, if we define~$r_k(x) \defeq \pi_{k+1}(x)/\pi_k(x)$, then
\begin{equation}
    M^{-1} \leq \inf_{\T^d} r_k \leq \sup_{\T^d} r_k \leq M \,.
\end{equation}
Thus, we obtain that for each~$k\in\set{0, \ldots, N}$ and~$i\in \set{1,2}$, 
\begin{align}\label{e:pik-k+1-overlap}
    \int_{\Omega_i} \min\set*{\pi_k(z), \pi_{k+1}(z)} dz &= \int_{\Omega_i} \min\set*{1, r_k(z)} \pi_k(z) dz \geq M^{-1}\pi_k(\Omega_i)\,,
\end{align}
and for each~$k\in\set{1, \ldots, N+1}$,
\begin{equation}\label{e:pik-k-1-overlap}
    \int_{\Omega_i} \min\set*{\pi_k(z), \pi_{k-1}(z)} dz = \int_{\Omega_i} \min\set*{1, r_{k-1}^{-1}(z)}\pi_k(z) dz \geq M^{-1}\pi_k(\Omega_i)\,.
\end{equation}
Combining~\eqref{e:pik-k+1-overlap} and~\eqref{e:pik-k-1-overlap} yields~\eqref{e:delta-swc-lb}.
\end{proof}

\begin{lemma}\label{l:general-eh-gap}
    There exists a dimensional constant~$\hat c_d>0$ such that for any~$\epsilon > 0$ and~$h \in (0,1]$,  
   \begin{equation}\label{e:general-eh-gap}
        \gap(T_{h, \epsilon})\geq \hat c_d\exp\paren*{-\frac{2\norm{U}_{L^\infty}}{\epsilon} }h^2 \,.
    \end{equation}
    Here, $T_{h, \epsilon}$ is the lazy Metropolis random walk with step size~$h$ and stationary density~$\pi_\epsilon$, defined as in~\eqref{e:Thpi-def}.
\end{lemma}

\begin{proof}
    We first notice that
    \begin{equation}\label{e:first-level-stationary-density-bounds}
        c(\epsilon) \defeq \exp\paren*{-\frac{\norm{U}_\infty}{\epsilon}}\leq \tilde \pi_\epsilon \leq 1\,,
    \end{equation}
    and hence, using Holley-Stroock Lemma~\ref{l:holley-stroock} yields that
    \begin{equation}\label{e:Phe-Ph-compare}
        \gap(T_{h, \epsilon}) \geq c(\epsilon)^2 \gap(T_{h, \infty})\,,
    \end{equation}
    where~$T_{h, \infty}(x,\cdot)$ is the lazy Metropolis random walk with step size~$h$ and the Lebesgue stationary distribution.
    Since~$T_{h, \infty}$ has the Lebesgue stationary distribution, it always accepts the proposed move with probability~$1/2$ and hence the local conductance~$l$ satisfies
    \begin{equation}
        l = \inf_{x\in \T^d} T_{h, \infty}\paren*{x, \T^d\setminus\set{x}} = \frac{1}{2}\,. 
    \end{equation}
    Applying~\cite[Corollary 3.3]{LovaszSimonovits93} with~$t=1/2$, and~$\theta = \hat c_d h \leq \min\set*{1, \hat c_d}$ for some dimensional constant~$\hat c_d$, we obtain
    \begin{equation}
        \gap(T_{h, \infty}) \geq \hat c_d h^2\,,
    \end{equation}
    and combining this with~\eqref{e:Phe-Ph-compare} implies~\eqref{e:general-eh-gap}.
\end{proof}

Finally, combining these estimates with the results established in the previous sections, we obtain Theorem~\ref{t:main-theorem}.

\begin{proof}[Proof of Theorem~\ref{t:main-theorem}]
Let~$\hat\epsilon, \eta$ be as in Lemma~\ref{l:lyapunov}, and let~$\hat c_1$ and~$c_2(\eta, \hat\epsilon, 1)$ be as in Lemmas~\ref{l:sg-restricted} and~\ref{l:large-temperature-gap}, respectively. Set~$c_1 = \hat c_1 \eta^4$. 
Finally, define~$\hat C_\mathrm{BV}$ as in~\eqref{e:C_BV-def} and~$\hat c_d$ as in Lemma~\ref{l:general-eh-gap}, and set~$C_\mathrm{BV} = 5\hat C_\mathrm{BV}$ and~$c_d = 2^{-20} \hat c_d$.

We first note that Lemma~\ref{l:lyapunov} holds for any sufficiently small~$\hat\epsilon$ and~$\eta$. Hence, without loss of generality, we may assume that $\eta \hat\epsilon^2 \leq 1$. 
Choosing~$h_k$ as in~\eqref{e:sg-local-h-choice} ensures that for all~$\epsilon_k \in [\hat\epsilon, \bar \epsilon]$, we have $\eta\hat\epsilon^2 \leq h_k \leq 1$, so that Lemma~\ref{l:large-temperature-gap} applies. Therefore,
\begin{equation}\label{e:sg-local-large-temp}
    \inf_{\epsilon_k \in [\hat\epsilon, \bar\epsilon]} \gap(Q_{h_k, \epsilon_k}) \ge c_2\,.
\end{equation}

On the other hand, for all~$\epsilon_k \in [\underline\epsilon, \hat\epsilon]$, the choice~\eqref{e:sg-local-h-choice} implies that $0 < h_k = \eta\epsilon_k^2 \leq \eta \hat\epsilon^2 \leq 1$, so that Lemma~\ref{l:sg-restricted} applies. Hence,
\begin{equation}\label{e:sg-local-small-temp}
    \gap(Q_{h_k, \epsilon_k}) \geq \hat c_1 \eta^4 \epsilon_k^7 \geq c_1 \underline\epsilon^7\,.
\end{equation}

Combining the above bounds, and using the identity
\begin{equation}
    T_{h_k, \epsilon_k}\vert_{\Omega_1} = \frac{I + P_{h_k, \epsilon_k}\vert_{\Omega_1}}{2} = \frac{I+Q_{h_k, \epsilon_k}}{2}\,,
\end{equation}
together with the fact that laziness halves the spectral gap, and noting that the same argument applies to the other basin~$\Omega_2$, we obtain
\begin{equation}\label{e:lazyT-local-sg-lb-for-all-temp}
    \min_{i\in \set{1,2}}\inf_{\epsilon_k \in [\underline \epsilon, \bar\epsilon]} \gap(T_{h_k, \epsilon_k}\vert_{\Omega_i}) \geq \frac{1}{2}\min\set*{c_2, c_1\underline\epsilon^7}\,.
\end{equation}

Moreover, $T_{h_0, \epsilon_0}$ is a lazy (and hence non-negative definite) reversible Markov chain. Therefore,~\cite{MadrasRandall02} implies that
\begin{equation}
    \gap(\bar T_{h_0, \epsilon_0}) \geq \gap(T_{h_0, \epsilon_0})\,,
\end{equation}
where~$\bar T_{h_0, \epsilon_0}$ is the chain defined as in~\cite[Section 3, Equation (4)]{WoodardSchmidlerEA09}.

Combining this with Lemma~\ref{l:general-eh-gap} yields
\begin{equation}
    \gap(\bar T_{h_0, \epsilon_0}) \geq \hat c_d \exp\paren*{-\frac{2\norm{U}_\infty}{\epsilon_0}} h_0^2\,.
\end{equation}

Finally, combining this with~\eqref{e:gamma-swc-lb}, \eqref{e:delta-swc-lb}, and~\eqref{e:lazyT-local-sg-lb-for-all-temp}, and applying~\cite[Theorem 3.1]{WoodardSchmidlerEA09} to the decomposition~$\mathcal A = \set{\Omega_1, \Omega_2}$ with the number of wells~$J=2$ and the sequence of reversible Markov chains~$(T_{h_k, \epsilon_k}, \pi_{\epsilon_k})_{k=0}^N$, we obtain~\eqref{e:main-theorem-swc-gap}.
\end{proof}

\appendix
\section{Tools for bounding spectral gap}\label{a:tools-for-gap}

This is a well-known result from~\cite{HolleyStroock87}, \cite[Proposition 2.3]{MadrasPiccioni99} and~\cite[Lemma 3.3]{DiaconisSaloff96}, used to compare the spectral gaps of two Metropolis chains with the same symmetric proposal kernel. For the reader’s convenience, we reproduce the statement and proof here in a form suited to our setting, where the ratio of the unnormalized densities is controlled.

\begin{lemma}[Holley--Stroock]\label{l:holley-stroock}
Let~$(\mathcal X, \mathcal B, m)$ be a measure space, and let~$\tilde p_1, \tilde p_2$ be non-negative measurable functions such that $\tilde p_1, \tilde p_2 \in L^1(m)$ and do not vanish simultaneously. Define probability measures~$\pi_i$ by $d\pi_i/dm \propto \tilde p_i$, and let~$P_i(x,dy)$ be the transition kernels of Metropolis chains with stationary measures~$\pi_i$ and a common proposal kernel~$Q(x,dy)$.

Assume that the proposal kernel is symmetric in the following sense: for each~$x \in \mathcal X$, $Q(x,dy)$ admits a density~$q(x,y)$ on~$\mathcal X \setminus \{x\}$ with respect to~$m(dy)$, and $q(x,y) = q(y,x)$ for all~$x,y \in \mathcal X$.

If there exist constants~$a,b>0$ such that
\begin{equation}\label{e:p1p2-pointwise-bound}
    a \leq \frac{\tilde p_1(x)}{\tilde p_2(x)} \leq b\,, \quad \forall x \in \mathcal X,
\end{equation}
then
\begin{equation}\label{e:gap1gap2-bound}
   (b^{-1}a)^2\,\gap(P_2) \leq \gap(P_1) \leq (a^{-1}b)^2\,\gap(P_2)\,.
\end{equation}
\end{lemma}

\begin{proof}
Recall from~\eqref{e:gap-def} that
\begin{equation}\label{e:gap-def-Pi}
    \gap(P_i) = \inf_{f\in L^2(\pi_i)\setminus\{0\}}\frac{\mathcal E_i (f)}{\var_{\pi_i}(f)}\,,
\end{equation}
where
\[
\mathcal E_i(f) \defeq \frac{1}{2} \int_{\mathcal X}\int_{\mathcal X} |f(y)-f(x)|^2\, P_i(x,dy)\,\pi_i(dx).
\]
Let~$Z_i \defeq \int_{\mathcal X} \tilde p_i(x)\, m(dx)$ be the normalizing constants, so that $\pi_i(dx) = p_i(x)m(dx)$ with $p_i = \tilde p_i/Z_i$. By the definition of the Metropolis random walk and the symmetry of the proposal kernel, we have
\begin{equation}\label{e:energy-def}
    \mathcal E_i(f) = \frac{1}{2}\int_{\mathcal X}\int_{\mathcal X} |f(y)-f(x)|^2 q(x,y)\min\{p_i(y), p_i(x)\}\, m(dy)m(dx)\,.
\end{equation}

The bound~\eqref{e:p1p2-pointwise-bound} implies
\begin{equation}\label{e:p1p2-normalized-bound}
    aZ_2 \leq Z_1 \leq bZ_2\,, \qquad 
    b^{-1}a \leq \frac{p_1(x)}{p_2(x)} \leq a^{-1}b\,.
\end{equation}
Consequently,
\begin{equation}\label{e:e1e2-bound}
    b^{-1}a \leq \frac{\mathcal E_1(f)}{\mathcal E_2(f)} \leq a^{-1}b\,.
\end{equation}

Moreover, using the characterization
\[
\var_{\pi_i}(f) = \inf_{c\in \mathbb R} \int_{\mathcal X} (f-c)^2 p_i(x)\, m(dx),
\]
together with~\eqref{e:p1p2-normalized-bound}, we obtain
\begin{equation}\label{e:var1var2-bound}
    b^{-1}a\, \var_{\pi_2}(f) \leq \var_{\pi_1}(f) \leq a^{-1}b\, \var_{\pi_2}(f)\,.
\end{equation}

Combining~\eqref{e:gap-def-Pi}, \eqref{e:e1e2-bound}, and~\eqref{e:var1var2-bound} yields~\eqref{e:gap1gap2-bound}.
\end{proof}

We also prove Lemma~\ref{l:mrw-lyapunov} here. Our argument adapts part of the proof of~\cite[Theorem~1]{TaghvaeiMehta22}. In~\cite{TaghvaeiMehta22}, it is shown that a Lyapunov condition together with the existence of a small set yields a lower bound on the spectral gap of a reversible Markov chain, in terms of a coupling probability on the small set. We modify their argument to suit our setting of a Metropolis chain and to relate the spectral gap of the chain to that of its restriction. Similar connections in continuous time have been established, for example, in~\cite[Theorem~3.8]{MenzSchlichting14} and~\cite{BakryBartheEA08}.

\begin{proof}[Proof of Lemma~\ref{l:mrw-lyapunov}]
    By the definition of the spectral gap in~\eqref{e:gap-def} and the characterization of variance, it suffices to show that for any~$f\in L^2(\pi)$, there exists~$c$ such that 
    \begin{equation}\label{e:varf-lyapunov-ub}
        \int_\mathcal{X} (f-c)^2 \pi(x)m(dx) \leq \paren*{\frac{\alpha_1\lambda_1}{b_1+\alpha_1}}^{-1} \mathcal E (f)\,,
    \end{equation}
    where~$\mathcal E(f) \defeq \frac{1}{2} \int_\mathcal{X}\int_\mathcal{X} \abs*{f(y)-f(x)}^2 P(x,dy)\pi(x)m(dx)$.

    From~\cite[Equation~(7) and the discussion preceding it]{TaghvaeiMehta22}, we have that for any~$c\in \mathbb R$,
    \begin{align}\label{e:varf-ub}
        \lambda_1 \norm{f-c}_{L^2(\pi)}^2 \leq \ip{f, (I-P)f}_{L^2(\pi)} + b_1 \norm{(f-c)\one_K}_{L^2(\pi)}^2\,.
    \end{align}

    We choose~$c=\int f \,\pi_K(dx)$, where~$\pi_K(dx) = \pi(K)^{-1} \one_K(x)\pi(x)m(dx)$. Then
    \begin{align}\label{e:varf-K-ub}
        \norm{(f-c)\one_K}_{L^2(\pi)}^2 
        = \pi(K)\,\var_{\pi_K}(f_K) 
        \leq \pi(K)\,\gap(P\vert_K)^{-1}\mathcal E_K(f_K)\,,
    \end{align}
    where~$f_K = f\one_K$ and
    \[
    \mathcal E_K(f_K) 
    = \frac{1}{2} \int_K \int_K \abs*{f_K(y)-f_K(x)}^2 P\vert_K(x,dy)\pi_K(dx)\,.
    \]

    Let~$\E$ denote expectation under the joint law where~$X_0\sim \pi_K$, $X_1\sim P\vert_K(X_0,\cdot)$, and~$X_1^*\sim Q(X_0,\cdot)$, with~$Q$ the proposal kernel. Let~$\alpha_K$ and~$\alpha$ denote the Metropolis acceptance probabilities corresponding to~$\pi_K$ and~$\pi$, respectively. Using that~$\alpha_K(x,y)=\alpha(x,y)$ for all~$x,y\in K$, we obtain
    \begin{align}
        2\mathcal E_K(f)
        &= \E\brak*{\abs*{f_K(X_1)- f_K(X_0)}^2} \\
        &= \E\brak*{\abs*{f_K(X_1^*)- f_K(X_0)}^2 \alpha_K(X_0, X_1^*) \one_{\set{X_1^*\in K}}} \\
        &= \E\brak*{\abs*{f(X_1^*)- f(X_0)}^2 \alpha(X_0, X_1^*) \one_{\set{X_1^*\in K}}} \\
        &\leq \E\brak*{\abs*{f(X_1^*)- f(X_0)}^2 \alpha(X_0, X_1^*)} \\
        &= \pi(K)^{-1} \int_K \int_{\mathcal X} \abs{f(y)-f(x)}^2 \alpha(x,y) q(x,y)m(dy)\pi(x)m(dx) \\
        &\leq \pi(K)^{-1} \int_{\mathcal X} \int_{\mathcal X} \abs{f(y)-f(x)}^2 \alpha(x,y) q(x,y)m(dy)\pi(x)m(dx) \\
        \label{e:energy-K-ub}
        &= 2\pi(K)^{-1} \mathcal E(f).
    \end{align}

    Combining~\eqref{e:varf-ub}, \eqref{e:varf-K-ub}, and~\eqref{e:energy-K-ub} yields~\eqref{e:varf-lyapunov-ub}, completing the proof.
\end{proof}

\section{Regularities and properties of basins and projection}\label{a:regularities-of-basin-projection}

We prove Lemma~\ref{l:facts-on-boundaries} and Lemma~\ref{l:signed-distance-properties} together, since they are closely related.

\begin{proof}[Proof of Lemma~\ref{l:facts-on-boundaries} and~\ref{l:signed-distance-properties}]
We assume that~$U \in C^6(\T^d, \R)$.
Then~\cite[Chapter 2.7, The Stable Manifold Theorem, Remark 1]{Perko96} implies items~(\ref{i: boundary-regularity}) and~(\ref{i:tangent-normal-eigenvector}) in Lemma~\ref{l:facts-on-boundaries}.

Since~$\partial\Omega \in C^5$, it follows from~\cite[Chapter 14.6]{GilbargTrudinger01} (see also~\cite{KrantzParks81}) that there exists~$r_0>0$ such that the property~\eqref{e:distance-achieved} holds and moreover, the signed distance function~$d$ defined in~\eqref{e:signed-d-def} satisfies~\eqref{e:signed-d-regularity} and~\eqref{e:Dd-normal-identity}.

Combining the two identities in~\eqref{e:Dd-normal-identity}, we obtain
\begin{equation}
    \xi(x) = x + d(x) Dd(x)^T\,, \quad \forall x\in \Gamma_{r_0}\,.
\end{equation}
Since~$d \in C^5(\Gamma_{r_0})$ and~$Dd \in C^4(\Gamma_{r_0})$, it follows that~$\xi \in C^4(\Gamma_{r_0})$, which proves~\eqref{e:pi-regularity}. This completes the proof.
\end{proof}

\begin{lemma}[Unique Projection along Normals]
\label{l:unique_projection}
Let $\partial\Omega \subset \mathbb{T}^d$ be a compact $C^k$ manifold with $k \ge 3$. There exists a uniform constant $t_0 > 0$ such that for any $x \in \partial\Omega$ and any real number $t$ satisfying $|t| \le t_0$, the unique closest point on $\partial\Omega$ to the point $y = x + t n(x)$ is $x$ itself. That is,
\[ \xi(x + t n(x)) = x \]
where $\xi$ is the closest-point projection map.
\end{lemma}

\begin{proof}
Fix a point $x \in \partial\Omega$. Because $\partial\Omega$ is a $C^k$ manifold, it can be represented locally as a graph over its tangent space $T_x\partial\Omega$. For any point $z \in \partial\Omega$ in a sufficiently small neighborhood of $x$, we can uniquely decompose $z$ as
\[ z = x + v + h(v)n(x) \]
where $v \in T_x\partial\Omega$ is a tangent vector, $n(x)$ is the unit outward normal at $x$, and $h$ is a $C^k$ height function. Since $T_x\partial\Omega$ is tangent to the manifold at $x$, we have $h(0) = 0$ and $Dh(0) = 0$. By Taylor's theorem, there exists a constant $C_x > 0$ bounding the second derivatives such that $|h(v)| \le C_x |v|^2$. Because $\partial\Omega$ is compact, its principal curvatures are globally bounded, so we can choose a uniform constant $C = \max_{x \in \partial\Omega} C_x$ independent of $x$.

Now, let $y = x + t n(x)$. The squared distance from $y$ to $x$ is $|y - x|^2 = |t n(x)|^2 = t^2$. 

To show that $x$ is the strictly unique closest point to $y$, we evaluate the squared distance from $y$ to any other nearby point $z \neq x$ on the manifold (so $v \neq 0$):
\begin{align*}
    |y - z|^2 &= |(x + t n(x)) - (x + v + h(v)n(x))|^2 \\
    &= |-v + (t - h(v))n(x)|^2.
\end{align*}
Since $v$ is orthogonal to $n(x)$, we apply the Pythagorean theorem:
\begin{align*}
    |y - z|^2 &= |v|^2 + (t - h(v))^2 \\
    &= |v|^2 + t^2 - 2th(v) + h(v)^2 \\
    &\ge |v|^2 + t^2 - 2|t||h(v)|.
\end{align*}
Substituting the uniform curvature bound $|h(v)| \le C|v|^2$, we obtain:
\begin{align*}
    |y - z|^2 &\ge t^2 + |v|^2 - 2|t|C|v|^2 \\
    &= t^2 + |v|^2 (1 - 2|t|C).
\end{align*}
If we define $t_0 < \frac{1}{2C}$, then for any $t$ such that $|t| \le t_0$, we have $1 - 2|t|C > 0$. Because $z \neq x$ implies $|v| > 0$, the second term is strictly positive. Therefore,
\[ |y - z|^2 > t^2 = |y - x|^2. \]
Thus, $x$ is the strictly unique minimizer of the distance to $y$, concluding the proof.
\end{proof}

\bibliographystyle{halpha-abbrv}
\bibliography{gautam-refs1,gautam-refs2,preprints, references}
\end{document}